\documentclass[11pt]{article}
\usepackage{mydef2col}
\usepackage{markArticle}

\title{Flux-Corrected Diagonal Frog: second order and positivity at all time steps}

\shorttitle{Flux-Corrected Diagonal Frog}

\author{
    \authorstyle{ Andrey Itkin}
    \newline \newline
    \institution{FREE department, Tandon School of Engineering, New York University, email: \url{aitkin@nyu.edu}} \\
}

\date{\today}

\begin{document}

\maketitle

\lettrineabstract{By Godunov's theorem, linear second-order finite-difference schemes for the Fokker-Planck equation cannot preserve positivity. The Diagonal Frog (DF) framework previously bypassed this barrier using eventual positivity, but required a strict minimum time step. This paper resolves the small-step limitation using a nonlinear extension of the DF solvers. We split the second-order directional operator into a monotone M-matrix core and an antidiffusive flux correction. A Zalesak-type limiter is then applied iteratively within the implicit banded solve. The resulting Flux-Corrected DF (FCDF) schemes (variants A and B) are unconditionally positive across all time steps. Because the limiter acts on fluxes rather than point values, these schemes conserve discrete mass exactly and maintain second-order accuracy. Crucially, the limiter activates only within unresolved layers. This ensures the global $L_1$ convergence remains second-order uniformly in the cell P\'eclet number, avoiding the first-order degradation seen in the Chang-Cooper scheme. The method's Picard iteration is contractive under a purely convective step restriction. To support arbitrary step sizes, we develop an active-set reformulation. This solves the system using a semismooth Newton iteration, where computational cost scales only with the number of nodes where positivity binds. Finally, we introduce a defect-corrected time stepping approach that restores second-order time accuracy. Numerical experiments on Ornstein-Uhlenbeck and advection-dominated benchmarks confirm our claims.
}

\section{Introduction}

A linear finite-difference (FD) scheme that preserves nonnegativity of its solution is monotone, and by Godunov's theorem \cite{Godunov1959} a linear monotone scheme is at most first-order accurate. For the Fokker--Planck equation (FPE) this barrier is not a technicality: the solution is a probability density, nonnegativity is a physical requirement rather than a cosmetic one, and second-order accuracy in both space and time is what makes the computation affordable. In financial applications the same requirement appears as the absence of butterfly arbitrage: the Dupire forward equation for the option-implied transition density is an FPE, and calibration of local volatility surfaces rests on that density remaining nonnegative (see e.g., \cite{ItkinLocalVol} and references therein).

The two companion papers of this series confront the barrier from the large-step side. In \cite{ItkinDF2026} the directional operators of the Diagonal Frog (DF) discretization are eventually nonnegative matrices (EM), and the matrix exponential $e^{\Delta t L}$ is entrywise nonnegative for all steps above an explicit threshold $\tau_0$. In \cite{ItkinKazbek2026ADI} the exponential is replaced by resolvents, and the backward Euler map $(\mathcal I-\gamma L)^{-1}$ is proved nonnegative on the mirrored one-sided window $\gamma\ge\gamma_0$. In both cases positivity is a large-step phenomenon and provably fails as $\Delta t\to0$ at fixed mesh. This is Godunov's barrier seen from inside the linear theory: below the threshold, no repair of a linear second-order scheme is possible even in principle.

The only remaining escape is nonlinearity, and the DF framework already contains the germ of one. The factorized mixed-derivative solver of \cite{Itkin3D,ItkinDF2026} monitors the sign of the running right-hand side at runtime and terminates its Picard iteration at the last nonnegative iterate. Stopping at a data-dependent moment is a nonlinear rule; it is what makes the solver positive in practice without contradicting the Godunov's theorem. The present paper makes that nonlinearity explicit, local and quantitative, and turns it into a scheme with proved properties rather than a monitoring heuristic.

The problem addressed in this work resides at the intersection of two computational literatures that have historically developed in isolation. The first tradition encompasses positivity-preserving discretizations specific to the FPE. Its foundation lies in the Chang--Cooper (CC) scheme, \cite{ChangCooper1970}, an exponential-fitting construction closely related to the Scharfetter--Gummel flux used in semiconductor device modeling, \cite{ScharfetterGummel1969}. Discretizations within this family, including finite-volume descendants that achieve discrete entropy dissipation and exactly preserve steady states, \cite{BessemoulinChatard2012}, have been extensively analyzed and extended for Fokker--Planck operators in  \cite{LarsenEtAl1985,BuetDellacherie2010,MohammadiBorzi2015,PareschiZanella2018}. However, because these structure-preserving schemes are linear in the solution, they are governed by Godunov's theorem. Consequently, positivity is achieved either through severe restrictions on the time step and spatial mesh, or via a reduction to first-order accuracy in the convection-dominated limit. For the CC scheme, this accuracy reduction is global and unconditional, as quantified in \cref{sec:peclet}.

The second literature comprises the flux-corrected transport (FCT) tradition, which fundamentally accepts nonlinearity as the necessary cost for achieving second-order accuracy. Originating with the foundational work \cite{BorisBook1973} and Zalesak's multidimensional limiter, \cite{Zalesak1979}, this approach evolved through the total variation diminishing (TVD) frameworks, \cite{Harten1983,Sweby1984}. In the context of finite elements, these concepts have been adapted into algebraic flux correction, \cite{KuzminTurek2002,Kuzmin2009,FCT2012} with rigorous mathematical analyses addressing well-posedness and convergence in  \cite{BarrenecheaJohnKnobloch2016,BarrenecheaJohnKnobloch2024}. A critical distinction within this field separates TVD-type limiting, which strictly constrains local extrema, from positivity-only limiting, which constrains only the sign and thereby retains full accuracy at smooth extrema, \cite{ZhangShu2010,ZhangShu2011}. The spatial discretization proposed herein aligns with this latter tradition, but introduces two distinct departures: the limiter operates within an implicit banded M-matrix sweep rather than on an explicit update, in the factorized-solver idiom of \cite{Itkin3D,ItkinDF2026}, and mass conservation is inherited directly from the flux form during the implicit solve. As a result, exact conservation holds at every iteration and for all limiter values without requiring externally imposed constraints.

A third foundational strand concerns the positivity of the temporal discretization itself, delineating why the aforementioned barriers cannot be circumvented simply by employing a superior linear time-stepping method. It was established in \cite{BolleyCrouzeix1978} that any rational one-step method providing unconditional positivity for all linear positive systems is restricted to first-order accuracy, a temporal analogue to Godunov's spatial theorem. While strong-stability-preserving methods, e.g., \cite{GottliebShuTadmor2001,HundsdorferVerwer2003} respect this barrier by linking positivity to forward-Euler step restrictions, modified Patankar methods,  \cite{BurchardEtAl2003,KopeczMeister2018} bypass it nonlinearly for production-destruction systems. However, Patankar methods require a sign-separated decomposition of the generator, which is obstructed here by the wrong-signed band of the second-order convection stencil. In contrast, companion papers evade this barrier by restricting the class of matrices rather than the initial data. For eventually nonnegative generators,  \cite{NoutsosTsatsomeros2008,OleskyEtAl2009,DanersGlueckKennedy2016}, the exponential and resolvent are entrywise nonnegative above explicitly defined thresholds, see \cite{ItkinDF2026,ItkinKazbek2026ADI}. Building upon this, the coverage result established in \cref{sec:coverage} uniquely couples an FCT-type limiter below the threshold with an eventual-positivity window above it, spanning all possible step sizes, a synthesis without precedent in either literature.

To realize this methodology at large time steps, the proposed construction relies on an active-set solver detailed in \cref{sec:coverage}, which utilizes semismooth Newton methods standard in the theory of complementarity problems, \cite{QiSun1993,HintermuellerItoKunisch2002}. Their local superlinear convergence ensures that the nonlinear solve remains computationally affordable. The algorithm splits the second-order one-dimensional operator into a monotone core, comprising central diffusion and first-order upwind convection that form an M-matrix generator, and an antidiffusive correction formulated as a flux difference with a two-point numerical flux and zero column sums. The implicit directional step is executed via Picard iteration on this correction, where fluxes are constrained per interface by a Zalesak-type budget rule of \cite{BorisBook1973,Zalesak1979,FCT2012}.

We analyze two variants of this approach. Scheme FCDF-A leaves the correction unlimited and terminates at the last nonnegative iterate, acting as a literal analogue to the runtime rule in \cite{ItkinDF2026}. While computationally inexpensive, it rigidly discards the entire correction upon a single node's sign failure. Conversely, Scheme FCDF-B, the primary focus of this work, limits each flux individually, ensuring optimal correction retention. Crucially, Scheme FCDF-B guarantees unconditional positivity during the step, enforces a nonnegative right-hand side via the limiter while maintaining an inverse-positive core resolvent, and exactly conserves mass across all iterations and limiter states through the telescoping property of the limited fluxes.

The iteration contracts in the discrete $\ell_1$ norm, $\norm{\bm u}_1=\sum_i|u_i|$, the natural norm for probability densities. Its contraction factor is bounded by the purely convective ratio $c\,\gamma\bar\mu/h$, where $\gamma=\theta\Delta t$ is the implicit step parameter (the product of the scheme parameter $\theta$ and the time step $\Delta t$), $\bar\mu=\norm{\mu}_\infty$ is the maximum of the drift over the grid, $h$ is the spatial step, and $c$ is an absolute constant. There is no parabolic contribution since the diffusion resides entirely within the monotone implicit core $A_1$, the M-matrix generator that combines central diffusion with first-order upwind convection. The bound relies on the exact identity $\lVert(\mathcal I-\gamma A_1)^{-1}\rVert_1=1$, which follows directly from conservation and inverse-positivity. Finally, consistency is conditional, aligning precisely with Scheme B of \cite{ItkinDF2026}. Wherever the density satisfies the log-Lipschitz resolution condition, the budgets are not exceeded and the limiters equal one. In this regime, the converged iterate solves the full second-order system, and the limiter only activates where the density approaches the level of the truncation error.

The accuracy of the scheme at large P\'eclet numbers warrants distinct consideration. Its natural alternative, the CC scheme, \cite{ChangCooper1970}, exhibits a more severe form of degradation in this regime. The CC approach relies on exponential fitting, meaning the stencil is modified everywhere based on the cell P\'eclet number. In the convection-dominated limit, the fitted weights tend toward pure upwind differencing at every node. Consequently, the reduction to first-order accuracy is global, unconditional, and independent of the underlying solution. In contrast, the proposed FCDF degrades locally and adaptively. The limiter is driven by the sign of the right-hand side rather than the cell P\'eclet number, ensuring that second-order accuracy is preserved in every resolved region.

The reduction to the monotone core is restricted to unresolved layers of width $O(h/\mathrm{Pe}_h)$, and the price of a layer depends on what it contains. A tail layer, where the density itself is at the scale of the mesh, contributes $O(h)\cdot O(h)=O(h^2)$ to the $L_1$ error, so for smooth densities the global $L_1$ convergence remains second-order uniformly with respect to the P\'eclet number. An unresolved front of order-one height contributes $O(h)$, and \cref{sec:peclet} makes the dichotomy precise as a two-term error bound. The robustness in the smooth case is achieved through positivity-only limiting, \cite{ZhangShu2010,ZhangShu2011}. Unlike TVD limiting, which is at most first-order accurate at smooth extrema, the budget rule constrains only the sign, so it leaves smooth extrema untouched and triggers only where the density drops to the truncation-error level. Inside the layers the fallback to first order is the price the budgets exact for unconditional positivity, and it cannot be avoided by reverting to a linear discretization: by Godunov's theorem a positivity-preserving linear scheme is first-order everywhere, not only in the layers, so the nonlinear scheme confines to a few cells a loss that every linear competitor suffers globally.

Finally, the small-step schemes are designed to hand over to the linear theory at large step sizes, and the hand-over runs along two parallel tracks. The companion paper \cite{ItkinKazbek2026ADI} provides two linear positivity windows: the backward Euler resolvent is entrywise nonnegative for $\gamma\ge\gamma_0$ and is first-order in time, while the Pad\'e$(0,2)$ map is entrywise nonnegative for $\gamma\ge\gamma_r$ and is second-order in time and L-stable. On the nonlinear side, the Picard contraction requires $\gamma<\gamma_{\mathrm{pic}}=h/(2\bar\mu)$, with FCDF-B realizing the backward Euler step and its defect-corrected extension FCDF-DC, developed in \cref{sec:timeorder}, realizing a second-order step. Coverage of the step-size axis therefore reduces to two computable comparisons per mesh. If $\gamma_{\mathrm{pic}}\ge\min(\gamma_0,\gamma_r)$, the directional step admits a positive and exactly conservative realization for every step size. If moreover $\gamma_{\mathrm{pic}}\ge\gamma_r$, the realization can be chosen second-order in time as well as in space, wherever the density is resolved, again for every step size. The two conditions fare differently in practice. The first holds on the benchmark meshes, so positivity and conservation at all steps come from an explicit switchover between two mechanisms. The second fails on every mesh we tested, because the Pad\'e threshold is set by the deep-tail stationary mass, so full second order at large steps is delivered instead by the active-set reformulation described next.

If the regimes do not overlap, an active-set reformulation removes the step-size restriction from the nonlinear side. The key observation is that in one dimension, the full second-order system is banded and can be solved exactly in $O(n)$ operations. Consequently, the primary system inversion is highly efficient, and the nonlinearity is isolated entirely within the limiter. The procedure begins by solving the unlimited system and accepting the result if it is nonnegative. Otherwise, the limiter is resolved as a complementarity problem using a semismooth Newton iteration. The computational cost of this iteration is dictated by the size of the active set, the small subset of nodes where the positivity constraint binds, rather than by the magnitude of the time step. Since both stages of the defect-corrected step are solves of the same form, the active-set solver applies stage-wise, so the second-order branch\footnote{Throughout this paper, we use the term \emph{branch} to refer to a distinct variant, regime, or accuracy order of the scheme, not in the sense of bifurcation theory.} also extends beyond $\gamma_{\mathrm{pic}}$ when needed.

We present five main results. First, we introduce Schemes FCDF-A and FCDF-B, demonstrating their unconditional positivity, exact conservation across all limiter values, and convective contraction bounds (\cref{sec:limited}). Second, we establish a P\'eclet-uniform $L_1$ second-order accuracy statement and contrast this behavior with the Chang--Cooper scheme (\cref{sec:peclet}).

Third, we prove a two-part coverage corollary: under the computable condition $\gamma_{\mathrm{pic}} \ge \min(\gamma_0, \gamma_r)$ the directional step is positive and exactly conservative for every step size, and under the stronger condition $\gamma_{\mathrm{pic}}\ge\gamma_r$ it is in addition second-order in time and space wherever the density is resolved, pairing the FCDF schemes below the convective bound with the resolvent and Pad\'e$(0,2)$ windows of \cite{ItkinKazbek2026ADI} above it. The numerical study finds the first condition satisfied and the second violated on every mesh tested, so we also provide an active-set solver that closes the resulting gap, for the defect-corrected stages as well (\cref{sec:coverage}).

Fourth, we develop the defect-corrected scheme FCDF-DC, which restores second-order accuracy in time without requiring an explicit half-step (\cref{sec:timeorder}).
Finally, we present numerical experiments that systematically quantify each of these theoretical claims (\cref{sec:numerics}).

The remainder of the paper is organized to systematically build these results. \Cref{sec:setting} defines the one-dimensional setting and introduces the core-correction split. Building on this foundation, \cref{sec:limited} constructs and analyzes the limited iteration scheme, and \cref{sec:peclet} evaluates its accuracy at large P\'eclet numbers. \Cref{sec:timeorder} then develops the defect-corrected scheme FCDF-DC, which restores second order in time. \Cref{sec:coverage} combines the nonlinear schemes with the linear windows of \cite{ItkinKazbek2026ADI} into coverage of all step sizes and details the active-set solver. Finally, \cref{sec:numerics} presents our numerical experiments, and \cref{sec:conclusion} offers concluding remarks.

\section{The 1D setting and the core-correction split} \label{sec:setting}

We consider the one-dimensional FPE in divergence form,
\begin{equation}   \label{eq:fpe1d}
\frac{\partial p}{\partial t} = -\frac{\partial J}{\partial x}, \qquad
J(x,t) = \mu(x,t)\,p - \frac{\partial}{\partial x}\bigl[D(x,t)\,p\bigr],
\qquad D(x,t) \ge 0,
\end{equation}
where $p(x,t)$ is the probability density function, $t$ is time, $x$ is the spatial coordinate, $\mu(x,t)$ is the drift (or convection) coefficient, $D(x,t)$ is the diffusion coefficient, and $J(x,t)$ is the probability flux. Total probability changes only through boundary fluxes and is preserved exactly under zero-flux conditions. Following \cite{ItkinDF2026}, we assume $\mu(x,t)>0$ throughout this section. The case $\mu<0$ admits a fully symmetric treatment with the upwind direction reversed, and for sign-changing $\mu$ the stencil is applied directionally at each node.

The equation is discretized on a uniform grid $x_i = x_1+(i-1)h$, $i=1,\ldots,n$, by the second-order DF operator of \cite{ItkinDF2026}, which we now recall explicitly. The advective term is approximated by the second-order backward (upwind) difference,
\begin{equation}   \label{eq:FB2}
\frac{\partial}{\partial x}[\mu p]\bigg|_{x_i} \approx
\frac{3(\mu p)_i - 4(\mu p)_{i-1} + (\mu p)_{i-2}}{2h},
\end{equation}
consistent with $\mu>0$, and the diffusive term by the centered difference,
\begin{equation} \label{eq:SC2}
\frac{\partial^2}{\partial x^2}[D p]\bigg|_{x_i} \approx
\frac{(Dp)_{i+1} - 2(Dp)_i + (Dp)_{i-1}}{h^2},
\end{equation}
where $\mu_i=\mu(x_i,t)$ and $D_i=D(x_i,t)$. Combining the two in \eqref{eq:fpe1d} gives, at interior nodes,
\begin{equation}   \label{eq:A2row}
\dot{p}_i = \alpha_i\,p_{i-2} + \beta_i\,p_{i-1} + \gamma_i\,p_i + \delta_i\,p_{i+1},
\end{equation}
with coefficients
\begin{equation}   \label{eq:A2coeffs}
\alpha_i = -\frac{\mu_{i-2}}{2h}, \qquad
\beta_i  = \frac{D_{i-1}}{h^2} + \frac{2\mu_{i-1}}{h}, \qquad
\gamma_i = -\frac{2D_i}{h^2} - \frac{3\mu_i}{2h}, \qquad
\delta_i = \frac{D_{i+1}}{h^2}.
\end{equation}
These rows assemble the operator $A_2\in\mathbb{R}^{n\times n}$, with the near-boundary row, for which $x_{i-2}$ falls outside the grid, discretized by the first-order upwind stencil, and with the boundary rows modified so that the discrete fluxes through the first and the last interface vanish. Under this zero-flux closure the operator telescopes and $\bm{1}^\top A_2=0$, where $\bm{1}$ is the column vector of all ones, which is the discrete statement of mass conservation. The wrong-signed band responsible for the loss of monotonicity is visible directly in \eqref{eq:A2coeffs}: the entry $\alpha_i<0$ sits below the subdiagonal, so $A_2$ is not a Metzler matrix, and it is precisely this entry that makes $A_2$ an eventually nonnegative (EM) generator rather than an M-matrix one, \cite{ItkinDF2026}.

Both stencils admit a discrete flux form. Writing $g=\mu p$, the advective difference \eqref{eq:FB2} equals $[\hat g_{i+1/2}-\hat g_{i-1/2}]/h$ with the second-order upwind numerical flux $\hat g^{(2)}_{i+1/2}=\tfrac12(3g_i-g_{i-1})$, while the first-order upwind flux is $\hat g^{(1)}_{i+1/2}=g_i$. The diffusive difference \eqref{eq:SC2} is a flux difference of $[(Dp)_{i+1}-(Dp)_i]/h$. This telescoping structure is what the limiter of \cref{sec:limited} acts on.

We now split $A_2$ into a monotone core and an antidiffusive correction,
\begin{equation}\label{eq:coresplit}
A_2 \;=\; A_1 + C .
\end{equation}
The core $A_1$ combines the centered diffusion \eqref{eq:SC2} with first-order upwind convection. Its interior rows are
\begin{equation}\label{eq:A1coeffs}
\dot{p}_i = \beta'_i\,p_{i-1} + \gamma'_i\,p_i + \delta'_i\,p_{i+1},
\qquad
\beta'_i  = \frac{D_{i-1}}{h^2} + \frac{\mu_{i-1}}{h}, \qquad
\gamma'_i = -\frac{2D_i}{h^2} - \frac{\mu_i}{h}, \qquad
\delta'_i = \frac{D_{i+1}}{h^2},
\end{equation}
so $A_1$ is tridiagonal, Metzler, irreducible, and under the same zero-flux closure it is an M-matrix generator with $\bm{1}^\top A_1=0$. The correction $C=A_2-A_1$ contains no diffusion. Its interior rows follow by subtracting \eqref{eq:A1coeffs} from \eqref{eq:A2coeffs},
\begin{equation}\label{eq:corrC}
(C\bm p)_i \;=\; \frac{-\mu_{i-2}\,p_{i-2} + 2\mu_{i-1}\,p_{i-1} - \mu_i\,p_i}{2h},
\end{equation}
and, equivalently, $C$ is a flux difference with the two-point correction flux
\begin{equation}\label{eq:corrflux1}
(C\bm p)_i \;=\; -\,\frac{d_{i+1/2}(\bm p)-d_{i-1/2}(\bm p)}{h},
\qquad
d_{i+1/2}(\bm p) \;=\; \hat g^{(2)}_{i+1/2}-\hat g^{(1)}_{i+1/2}
\;=\; \frac{(\mu p)_i-(\mu p)_{i-1}}{2}\,.
\end{equation}
The correction flux is the discrete antidiffusion of the second-order upwind scheme: it is proportional to the one-sided difference of $\mu p$ and carries no factor of $D$. Under the zero-flux closure $C$ telescopes as well, $\bm{1}^\top C=0$, so the split \eqref{eq:coresplit} preserves the conservation structure term by term. In the mirrored case $\mu<0$ the two upwind nodes lie to the right of the interface and the correction flux becomes $d_{i+1/2}=\tfrac12\bigl[(\mu p)_{i+1}-(\mu p)_{i+2}\bigr]$, and for sign-changing drift the two orientations are combined interface by interface.

The spatial discretization turns \eqref{eq:fpe1d} into the linear system of ordinary differential equations $\dot{\bm p}(t)=A_2\,\bm p(t)$, where $\bm p(t)\in\mathbb{R}^n$ collects the nodal values of the density. Let $\Delta t>0$ be the time step, let $t_n=n\Delta t$, and let $\bm p^{\,n}$ denote the numerical approximation to $\bm p(t_n)$. The classical one-parameter $\theta$-scheme advances the solution by
\begin{equation}\label{eq:thetascheme}
\bigl(\mathcal I-\theta\Delta t\,A_2\bigr)\,\bm p^{\,n+1} \;=\; \bigl(\mathcal I+(1-\theta)\Delta t\,A_2\bigr)\,\bm p^{\,n},
\qquad \theta\in[0,1],
\end{equation}
where $\mathcal I$ is the identity matrix. The choice $\theta=1$ gives the backward Euler scheme, which is fully implicit and first-order accurate in time. The choice $\theta=\tfrac12$ gives the trapezoidal (Crank--Nicolson) scheme, which is second-order accurate in time.

Every member of this family requires the solution of one linear system per step, and the matrix of that system is the same up to the scalar in front of $A_2$. We therefore abstract the implicit directional step studied throughout the paper as
\begin{equation}\label{eq:implicitstep}
\bigl(\mathcal I-\gamma A_2\bigr)\,\bm p \;=\; \bm b, \qquad \gamma=\theta\Delta t,
\end{equation}
where $\bm b$ is the known right-hand side assembled from the previous time level. For backward Euler $\bm b=\bm p^{\,n}$, so $\bm b\ge0$ whenever the previous iterate is nonnegative. For the trapezoidal rule $\bm b=\bigl(\mathcal I+\tfrac{\Delta t}{2}A_2\bigr)\bm p^{\,n}$, and the nonnegativity of this explicit half is a separate question. It is taken up in \cref{sec:timeorder}, where a defect-corrected stepping restores second order in time using only solves of the form \eqref{eq:implicitstep} with nonnegative right-hand sides. In \cref{sec:limited,sec:peclet,sec:coverage} we accordingly assume $\bm b\ge0$ and study the solution of \eqref{eq:implicitstep}.

\section{The flux-limited Picard iteration} \label{sec:limited}

In this section we consider two possible schemes of solving \eqref{eq:implicitstep}

\myparagraph{Scheme FCDF-A: the global stopping rule.}
The step \eqref{eq:implicitstep} is realized by Picard iteration on the correction,
\begin{equation}\label{eq:picard1d}
\bigl(\mathcal I-\gamma A_1\bigr)\,\bm p^{[k+1]} \;=\; \bm b \;+\; \gamma\, C\,\bm p^{[k]},
\qquad \bm p^{[0]}=\bm b ,
\end{equation}
each sweep being one banded M-matrix solve at $O(n)$ cost. Scheme FCDF-A terminates the iteration at the last iterate whose right-hand side is entrywise nonnegative, in exact analogy with the runtime rule of the mixed-derivative solver in \cite{ItkinDF2026}. Since $(\mathcal I-\gamma A_1)^{-1}\ge0$ unconditionally, every accepted iterate is nonnegative; since $\bm1^\top A_1=\bm1^\top C=0$, every iterate conserves mass exactly. The scheme is simple and nearly free to implement, but the stopping rule is global: a sign failure at a single node discards the correction everywhere in that sweep. Scheme FCDF-B removes this defect.

\myparagraph{Scheme FCDF-B: per-interface limiting.}
Introduce per-interface limiters $\theta_{i+1/2}\in[0,1]$ and the limited correction
\begin{equation}\label{eq:limcorr}
\bigl(C_\theta\,\bm p\bigr)_i \;=\; -\,\frac{\theta_{i+1/2}\,d_{i+1/2}(\bm p)-\theta_{i-1/2}\,d_{i-1/2}(\bm p)}{h},
\end{equation}
and iterate
\begin{equation}\label{eq:picardlim}
\bigl(\mathcal I-\gamma A_1\bigr)\,\bm p^{[k+1]} \;=\; \bm b \;+\; \gamma\, C_\theta\bigl(\bm p^{[k]}\bigr),
\qquad \bm p^{[0]}=\bm b .
\end{equation}
The limiters are chosen by the two-point Zalesak rule \cite{Zalesak1979}: each interface correction transfers the amount $\gamma\,\theta_{i+1/2}\,d_{i+1/2}/h$ between the two adjacent nodes, decreasing one (the donor) and increasing the other by exactly the same amount. Denote by $j(i{+}1/2)$ the donor node of interface $i+1/2$, i.e. $j=i$ if $d_{i+1/2}>0$ and $j=i+1$ if $d_{i+1/2}<0$, and by $\kappa_j\le2$ the number of interfaces drawing on node $j$ in the current sweep. Then
\begin{equation}\label{eq:zalesak}
\theta_{i+1/2} \;=\; \min\Bigl(1,\;\frac{h\,b_{j(i+1/2)}}{\kappa_{j}\;\gamma\,\bigl|d_{i+1/2}\bigl(\bm p^{[k]}\bigr)\bigr|}\Bigr),
\end{equation}
so that the total withdrawal from any node never exceeds half of its available budget per adjacent interface, which is the even budget split $\kappa_j\equiv2$. The limiters are recomputed from the current iterate at every sweep. This costs nothing analytically: the limited flux is then the clamp of the unlimited flux onto the fixed interval set by the caps, a clamp is nonexpansive, and the fixed-point analysis of \cref{prop:schemeBp} covers the resulting nonlinear sweep map with the same contraction constant as the linear one. The cheaper variant, with limiters computed once on the first sweep and frozen thereafter, is discussed in the remark after.

\begin{proposition}[Scheme FCDF-B]\label{prop:schemeBp}
Let $\bm b\ge0$, let $A_1$ be the irreducible M-matrix core of \eqref{eq:coresplit}, let $\bar\mu:=\norm{\mu}_\infty$, and let the limiter \eqref{eq:zalesak} use the even budget split $\kappa_j\equiv2$ and be recomputed from the current iterate at every sweep. Then, for every $\gamma>0$:
\begin{enumerate}
\item[\textup{(i)}] (Unconditional positivity) Every iterate of \eqref{eq:picardlim} satisfies $\bm p^{[k]}\ge0$, $k\ge0$. No window and no threshold in $\gamma$ is required: the limiter caps the withdrawal from every node by half of its budget per adjacent interface, so the right-hand side is nonnegative for any input, and $(\mathcal I-\gamma A_1)^{-1}\ge0$ for every $\gamma>0$.
\item[\textup{(ii)}] (Exact conservation for every limiter value) $\bm1^\top\bm p^{[k]}=\bm1^\top\bm b$ for all $k\ge0$, irrespective of the limiter values. Limiting acts on fluxes, not on point values, so every limited interface flux enters two rows with opposite signs and the limited correction telescopes, while the left-hand side has unit column sums.
\item[\textup{(iii)}] (Contraction) $\bigl\|(\mathcal I-\gamma A_1)^{-1}\bigr\|_1=1$ exactly, by conservation and inverse-positivity. The limited flux is the clamp of the unlimited flux onto a fixed interval containing zero, hence nonexpansive, and one sweep of the iteration is a Lipschitz map with constant $q \;=\; \frac{2\gamma\bar\mu}{h}$. For $q<1$, i.e. $\gamma<\gamma_{\mathrm{pic}}:=h/(2\bar\mu)$, the sweep map is a contraction in $\ell_1$ and the iteration converges geometrically to a unique nonnegative fixed point $\bm p^{*}$. We call $\gamma_{\mathrm{pic}}$ the Picard contraction threshold. It coincides in form with a convective Courant-number bound, but it is not a stability or positivity restriction: parts (i) and (ii) hold for every $\gamma$, and the threshold bounds only the convergence of the inner iteration, which \cref{ssec:activeset} removes.

\item[\textup{(iv)}] (Conditional consistency) At every node whose two adjacent caps are slack at the fixed point, $\gamma\,|d_{i\pm1/2}(\bm p^{*})|\le h\,b_{j}/2$ with $j$ the respective donor node, the limiters equal one and the corresponding row of the fixed-point system coincides with the full second-order system \eqref{eq:implicitstep}. If the caps are slack at every interface, $\bm p^{*}$ is exactly the unlimited second-order solution. The slack condition holds, in particular, wherever the fixed point satisfies the log-Lipschitz resolution condition of \cite{ItkinDF2026} and the budget is comparable to the density, so the limiter can activate only at nodes where $\bm p^{*}$ is at the level of the local truncation error.
\end{enumerate}
\end{proposition}

\begin{proof}[Proof sketch]
The core resolvent $(\mathcal I-\gamma A_1)^{-1}$ is entrywise positive with unit column sums, so its induced $\ell_1$ norm equals one exactly. The limited flux produced by \eqref{eq:zalesak} is the clamp of the unlimited flux onto a fixed interval set by the budget caps, an interval containing zero because $\bm b\ge0$. Positivity then follows by counting withdrawals: each of the two interfaces adjacent to a node can remove at most half of that node's budget, so the right-hand side of every sweep is nonnegative for any input, and the inverse-positive core preserves the sign. Conservation follows because every clamped interface flux enters the divergence in exactly two rows with opposite signs, so the limited correction telescopes for every limiter value. For the contraction, the clamp is nonexpansive, so one sweep is Lipschitz with the same constant as the unlimited linear map, and the flux coefficients bound that constant by $2\gamma\bar\mu/h$ with no contribution from $D$. Consistency at slack caps is read off row by row from the fixed-point equation, and the log-Lipschitz estimate bounds the correction flux by a multiple of $h$ times the local density, which keeps the caps slack wherever the budget is comparable to the density. The full proof is given in \cref{appSchemeB}.
\end{proof}

The following remarks provide further clarity on this proposition. First, the $\beta$-shift device of the mixed-derivative solver in \cite{Itkin3D,ItkinDF2026} does not transfer here: the iterated term $\gamma C\bm p^{[k]}$ grows linearly in $\gamma$ while the resolvent gain is pinned to exactly one by conservation, so no diagonal shift can improve the rate. The removal of the restriction at large $\gamma$ is achieved instead by the active-set solver of \cref{ssec:activeset} and, independently, by the linear resolvent branch of \cref{sec:coverage}.

Second, freezing the limiters after the first sweep makes each subsequent sweep linear and marginally cheaper, and the contraction bound of \cref{prop:schemeBp}(iii) applies to it as a special case, since $\norm{C_\theta}_1\le\norm{C}_1$ uniformly in $\theta$. What freezing loses is the positivity guarantee for the intermediate iterates: the caps are calibrated to the fluxes of the first iterate, and nothing prevents a later iterate from violating them, so \cref{prop:schemeBp}(i) applies to the frozen variant only at the first sweep and, when the caps happen to remain slack, at the limit. The recomputed variant is therefore the default throughout the paper.

Third, the limiter \eqref{eq:zalesak} is flux-corrected transport \cite{BorisBook1973,Zalesak1979} and algebraic flux correction \cite{FCT2012} transplanted into the factorized M-matrix solver idiom of the DF papers, with two differentiators: exact conservation is inherited from the flux form through the implicit sweep, not imposed a posteriori, and the scheme is designed to hand over to the linear resolvent theory of \cite{ItkinKazbek2026ADI} at large steps (\cref{sec:coverage}). Modified Patankar schemes achieve unconditional positivity at second order for production--destruction systems \cite{BurchardEtAl2003}, but require a sign-separated decomposition of the generator, which the wrong-signed band of $A_2$ obstructs; a shift large enough to separate signs reintroduces the same difficulty in another form.

\subsection{Accuracy at large P\'eclet numbers} \label{sec:peclet}

The CC scheme of \cite{ChangCooper1970} and the FCDF-B scheme degrade in different ways. CC is exponential fitting: the stencil is modified everywhere according to the local cell P\'eclet number $\mathrm{Pe}_h=\mu h/D$, and as $\mathrm{Pe}_h\to\infty$ the fitted weights tend to pure upwind at every node. The degradation to first order is therefore global, unconditional, and independent of the solution.

The FCDF-B scheme degrades locally and adaptively. The limiter is driven by the sign of the right-hand side, not by $\mathrm{Pe}_h$, so the reduction to the monotone core is confined to unresolved layers of physical width $O(D/\mu)=O(h/\mathrm{Pe}_h)$, at most a few cells, while away from the layers the scheme coincides with the unlimited second-order discretization by \cref{prop:schemeBp}(iv).

To quantify the price of the layers we first fix the terminology. A front is a region where the density transitions between two levels over the physical width set by the convection-diffusion balance, $\ell=D/\mu$. The front is resolved when $\ell\ge h$ and unresolved when $\ell<h$, which is precisely the condition $\mathrm{Pe}_h>1$, so that the entire transition falls inside a single cell and adjacent nodal values differ by the full height of the jump. The height of the front is the scale of the levels it connects. A front of height $O(h)$ occurs in the tails of a smooth density, where the resolution condition of \cref{prop:schemeBp}(iv) fails only because the density itself has dropped to the truncation scale. A front of height $O(1)$ connects two order-one levels, and there the limiter has no small parameter to hide behind: the antidiffusive flux at the offending interfaces exceeds any admissible budget, the limiter clamps it, and the scheme locally falls back to the monotone core.

A concrete example arises already for the single resolvent problem \eqref{eq:implicitstep}. Take $\mu=1$, $D=10^{-4}$ and $h=10^{-2}$ on the unit interval, so that $\mathrm{Pe}_h=100$, and let $\bm b$ be the nodal values of the indicator of $[0,\tfrac12]$. The exact solution of $(1-\gamma\mathcal L)u=b$ is a smoothed step: the discontinuity of $b$ is displaced downstream by a distance of order $\gamma\mu$, smaller than one cell for $\gamma\le h/(2\bar\mu)$, and smeared over a width of order $\max(\sqrt{\gamma D},\,D/\mu)\approx10^{-3}$, which is below the mesh size. On the grid the front appears as a pair of adjacent nodes carrying values close to one and close to zero. The budgets at the two or three interfaces around this pair are exceeded, the limiter activates, and the front is resolved at first order over a few cells. The resulting $L_1$ cost is the height times the width, $O(1)\cdot O(h)=O(h)$. This is the situation the edges of a compactly supported density create in the advection-dominated benchmark. The OU benchmark, by contrast, contain no such front: its profile is Gaussian-like, the resolution condition fails only in the tails, and every layer there has height $O(h)$, costing $O(h)\cdot O(h)=O(h^2)$.

A third regime deserves separate mention because it escapes pointwise error analysis altogether. Under zero-flux boundary conditions with the drift directed into a wall, the density accumulates in a boundary layer of width $\ell=D/\mu$ with peak height of order $1/\ell$, carrying an order-one probability mass. In the numbers above this is a spike of height $10^{4}$ and width $10^{-4}$. When $\ell<h$ this profile is a near-delta that no fixed-order discretization represents in a nodal $L_1$ norm: the discrete solution necessarily spreads the unit mass over at least one cell of width $h$, while the exact density concentrates it in $\ell$, and the nodal comparison saturates at $O(1)$ regardless of the scheme. This is not a failure of the FCDF construction but of the metric, since the exact solution leaves the space where nodal comparison is meaningful. The appropriate assessment there is the mass per cell or a weak norm, and the numerical section reports the boundary-layer cases in that form.

The taxonomy is therefore threefold. Tail layers of height $O(h)$ cost $O(h^2)$ and are the generic case for smooth Fokker--Planck densities. Interior unresolved fronts of height $O(1)$ cost $O(h)$, which is the monotone floor at a front and cannot be improved by any scheme in the class considered here. Sub-cell boundary accumulation falls outside nodal $L_1$ assessment and is measured by cellwise mass. The proposition below makes the first two regimes precise as a two-term error bound: a smooth contribution of size $O(h^2)$, uniform in the P\'eclet number, plus a layer contribution proportional to the number of layer interfaces and to the solution scale inside them. The third regime is excluded by the assumption of a bounded layer strength and is treated separately in \cref{sec:numerics}.

The two failure modes are distinguished by where the loss appears rather than by its size at any single interface, and this has a consequence for how the distinction can be observed. At an unresolved front of order-one height both schemes sit on the monotone floor, so their observed orders are close and a comparison there says little about the mechanism. The distinction is measurable on a smooth, strictly positive solution at large P\'eclet number. There the FCDF budgets are never exceeded, the limiter stays inactive, and the scheme retains its full order however large $\mathrm{Pe}_h$ becomes. The CC stencil is modified at every node regardless, because the fitting weight depends on the cell P\'eclet number and not on the solution, so its order drops. \Cref{ssec:numPeclet} accordingly reports two experiments, a smooth one for the contrast and a front one for the floor.

We compare the FCDF-B fixed point with the exact solution of the resolvent problem. Let $u$ solve $(1-\gamma\mathcal{L})u=b$ with the zero-flux boundary conditions, where $\mathcal{L}u=-(\mu u)'+(Du)''$, let $\bm u$ collect its nodal values, let $\hat{\bm p}$ denote the unlimited solution of \eqref{eq:implicitstep}, and let $\bm p^{*}$ be the FCDF-B fixed point of \cref{prop:schemeBp}. Errors are measured in the discrete norm $\norm{\bm v}_{1,h}=h\sum_i|v_i|$.

We formalize the layer structure through three assumptions. First, the step satisfies $\gamma\le\kappa\,\gamma_{\mathrm{pic}}$ for a fixed $\kappa\in(0,1)$. Second, there is a set $\mathcal A$ of at most $N_L$ interfaces, the layer interfaces, containing the boundary interfaces, the near-boundary first-order rows, and the active set of the limiter at $\hat{\bm p}$, such that on every cell not touching $\mathcal A$ the exact solution obeys the P\'eclet-uniform bound $|(\mu u)''|+|(Du)'''|+|J'''|\le M_S$, with $J$ the exact flux of \eqref{eq:fpe1d}. Third, the layer strength $S_L$ bounds the solution scale at the layer interfaces, $|d_{i+1/2}(\hat{\bm p})|\le\bar\mu S_L$ and $|J|\le\bar\mu S_L$ on $\mathcal A$, and it is precisely this boundedness that excludes the sub-cell boundary-accumulation regime of the taxonomy above, where the density scale at the wall grows like $1/\ell=\mathrm{Pe}_h/h$ and no bound on $S_L$ uniform in the mesh and the P\'eclet number exists.

\begin{proposition}[P\'eclet-uniform $L_1$ error bound]\label{prop:peclet}
Under the three assumptions above there is a constant $C(\kappa)$, independent of $h$, $D$, and $\mathrm{Pe}_h$, such that
\begin{equation}\label{eq:pecletbound}
\norm{\bm p^{*}-\bm u}_{1,h} \;\le\; C(\kappa)\,\Bigl[\,M_S\,|\Omega|\;h^2 \;+\; N_L\,S_L\;h\,\Bigr],
\end{equation}
where $|\Omega|$ is the length of the computational interval. In particular:
\begin{enumerate}
\item[\textup{(a)}] If every layer carries density at scale $S_L\le C_0h$, which is the positivity-limiting regime where the limiter activates only near vanishing density, then $\norm{\bm p^{*}-\bm u}_{1,h}=O(h^2)$ uniformly in $\mathrm{Pe}_h$.
\item[\textup{(b)}] If a layer carries an unresolved front of height $S_L=O(1)$, the bound degrades to $O(h)$, and the loss enters \eqref{eq:pecletbound} only through the $N_L$ layer interfaces, while the smooth contribution remains $O(h^2)$.
\end{enumerate}
\end{proposition}

\begin{proof}[Proof sketch]
The error splits into a consistency part, $\hat{\bm p}-\bm u$, and a limiter defect, $\bm p^{*}-\hat{\bm p}$. Stability costs nothing beyond the split: the resolvent of the core has $\ell_1$ norm one, so under $\gamma\le\kappa\gamma_{\mathrm{pic}}$ the error of the unlimited solve is bounded by $\gamma/(1-\kappa)$ times the truncation error. The truncation error is itself a flux difference. On smooth interfaces Taylor expansion of the upwind and centered numerical fluxes gives $O(h^2)$ flux errors, which sum to the first term of \eqref{eq:pecletbound}. On the $N_L$ layer interfaces the flux error is bounded crudely by the fluxes themselves, hence by $\bar\mu S_L$, which produces the second term. The limiter defect is controlled by the fixed-point perturbation bound for the contraction: it is proportional to the clamp residuals at $\hat{\bm p}$, which vanish off the active set and are bounded by $\bar\mu S_L$ on it, giving a contribution of the same size $N_LS_Lh$. All constants are independent of $D$ and $\mathrm{Pe}_h$ because the diffusion enters only through the core resolvent, whose norm is exactly one, and through $M_S$. The full proof is given in \cref{appPeclet}.
\end{proof}

Case (b) should be read as the expected price of monotone front resolution rather than a defect of the construction. A sub-cell front cannot be placed on the grid to better than one cell without tracking its position at sub-cell accuracy. Misplacing an order-one jump by one cell already produces an $O(h)$ error in $L_1$. The clamp confines the additional smearing to a few cells around the front. Godunov's theorem closes the linear escape route. A positivity-preserving linear scheme is first-order globally, so no linear discretization can improve on a loss that is here confined to the layer interfaces. For nonlinear positivity-preserving schemes we are not aware of a lower-bound theorem at this generality. We therefore report the $O(h)$ layer term as an observed floor rather than a proved one.

The bound \eqref{eq:pecletbound} concerns a single resolvent solve. Over many time steps a monotone scheme smears a front progressively. In the pure-advection limit the $L_1$ error of the monotone core degrades to the rate $O(h^{1/2})$, which is known to be sharp \cite{Kuznetsov1976,Sabac1997}. The limited scheme does not follow it down. Over two thousand steps at a front the observed rate of FCDF-B stays near one, while the monotone core settles near two thirds and drifts toward the $h^{1/2}$ floor, with an absolute error roughly eight times larger at the longest horizon. \Cref{ssec:numLongtime} reports the measurement.

\section{Second order in time} \label{sec:timeorder}

The Picard iteration \eqref{eq:picard1d} is an inner solver, not a time integrator: the index $k$ counts sweeps within a single implicit step, and the converged fixed point is the backward Euler solution of \eqref{eq:implicitstep} with $\theta=1$. The FCDF schemes of \cref{sec:limited} are therefore second-order in space but only first-order in time, and under the natural refinement $\Delta t\sim h$ the temporal error dominates. The trapezoidal rule would restore $O(\Delta t^2)$ but is unavailable in this framework: its explicit half $(\mathcal I+\tfrac{\Delta t}{2}A_2)$ has a diagonal that becomes negative under a parabolic-scale step, and flux limiting cannot repair a diagonal.

Two constructions restore second order in time without ever forming an explicit half, and they divide the step-size axis between them. At large steps the subdiagonal Pad\'e$(0,2)$ map of \cite{ItkinKazbek2026ADI}, $r_{02}(\Delta t\,A_2)=(\mathcal I-\Delta t\,A_2+\tfrac{\Delta t^2}{2}A_2^2)^{-1}$, is second-order, L-stable, fully implicit, and entrywise nonnegative as a linear map for all $\Delta t$ above the explicit threshold $\gamma_r$ established there, so inside that window it is the natural time stepper of the large-step branch of \cref{cor:coverage}. Below the window the limited machinery of \cref{sec:limited} does not extend to it: the denominator $1-z+z^2/2$ has the complex conjugate roots $1\pm\mathrm i$, and no factorization into real backward Euler factors exists, since real poles would require $a+b=1$ and $ab=\tfrac12$, so the M-matrix sign structure on which the analysis of \cref{prop:schemeBp} rests is unavailable for the quadratic solve. The small-step regime therefore needs its own second-order stepper built from limited backward Euler solves, and we construct it now.

\subsection{The defect-corrected step} \label{ssec:dcstep}

Freeze the coefficients at the midpoint $t_n+\Delta t/2$, so that $A_2$ is constant within the step. The scheme, which we call FCDF-DC, consists of a limited backward Euler predictor followed by one corrector solve of the same form,
\begin{align}
\bigl(\mathcal I-\Delta t\,A_2\bigr)\,\bm Y_0 \;&=\; \bm p^{\,n}, \label{eq:predictor}\\
\bigl(\mathcal I-\Delta t\,A_2\bigr)\,\bm p^{\,n+1} \;&=\; \bm p^{\,n} \;-\; \frac{\Delta t^2}{2}\,A_2^2\,\bm Y_0 , \label{eq:corrector}
\end{align}
both stages realized by the flux-limited iteration described below. The correction term costs one banded multiply, $\bm w:=A_2\bm Y_0$, followed by the flux assembly of $A_2\bm w$, so one time step costs two limited solves and one multiply.

The corrector \eqref{eq:corrector} is defect correction toward the trapezoidal rule, and the defect is available in closed form.
\begin{lemma}[Defect identity]\label{lem:defect}
If $\bm Y_0$ solves \eqref{eq:predictor} exactly, then its residual in the trapezoidal equation is
\begin{equation}\label{eq:defectid}
\Bigl(\mathcal I-\tfrac{\Delta t}{2}A_2\Bigr)\bm Y_0-\Bigl(\mathcal I+\tfrac{\Delta t}{2}A_2\Bigr)\bm p^{\,n}
\;=\;\frac{\Delta t^2}{2}\,A_2^2\,\bm Y_0 .
\end{equation}
\end{lemma}

\begin{proof}
The left-hand side equals $(\bm Y_0-\bm p^{\,n})-\tfrac{\Delta t}{2}A_2(\bm Y_0+\bm p^{\,n})$. By \eqref{eq:predictor}, $\bm Y_0-\bm p^{\,n}=\Delta t\,A_2\bm Y_0$ and $\bm p^{\,n}=\bm Y_0-\Delta t\,A_2\bm Y_0$, so $\tfrac{\Delta t}{2}A_2(\bm Y_0+\bm p^{\,n})=\Delta t\,A_2\bm Y_0-\tfrac{\Delta t^2}{2}A_2^2\bm Y_0$, and the two displays combine to \eqref{eq:defectid}.
\end{proof}

The exact trapezoidal solution is $\bm p_{\mathrm{TR}}=\bm Y_0-(\mathcal I-\tfrac{\Delta t}{2}A_2)^{-1}\tfrac{\Delta t^2}{2}A_2^2\bm Y_0$ by \cref{lem:defect}. Replacing the half-step resolvent by the full-step one, which perturbs the $O(\Delta t^2)$ correction by $O(\Delta t^3)$, gives $\bm p^{\,n+1}=\bm Y_0-(\mathcal I-\Delta t A_2)^{-1}\tfrac{\Delta t^2}{2}A_2^2\bm Y_0$, and multiplying through by $(\mathcal I-\Delta t A_2)$ and using \eqref{eq:predictor} yields \eqref{eq:corrector}. The reuse of the full-step matrix means both stages share one factorization.

On a scalar eigenvalue $z=\Delta t\,\lambda$ the two stages compose to the stability function
\begin{equation}\label{eq:Rfunc}
R(z)\;=\;\frac{1-z-z^2/2}{(1-z)^2} \;=\; 1+z+\frac{z^2}{2}+0\cdot z^3+O(z^4),
\end{equation}
so the local error is $-\tfrac{z^3}{6}+O(z^4)$ relative to $e^z$ and the scheme
is second-order in time. On the imaginary axis $|R(\mathrm iy)|^2 = (1+2y^2+y^4/4)/(1+2y^2+y^4)\le1$, and the only pole $z=1$ lies in the
right half-plane, so $R$ is A-stable. At infinity $R(\infty)=-\tfrac12$: by the
$r(\infty)$ criterion of \cite{ItkinKazbek2026ADI}, the unlimited two-stage map
has only a bounded positivity window as a linear map, so FCDF-DC is not a
competitor to Pad\'e$(0,2)$ at large steps. In the small-step regime its
positivity comes from the budgets, not from the sign of the linear map, as we now
show.

\subsection{The limited realization} \label{ssec:dclimited}

Both stages are solves of the form $(\mathcal I-\Delta t\,A_2)\,\bm p=\bm p^{\,n}+\bm s$ with a source $\bm s$ in exact flux form, $\bm s=\bm 0$ for the predictor and $\bm s=-\tfrac{\Delta t^2}{2}A_2\bm w$, $\bm w=A_2\bm Y_0$, for the corrector. Since $A_2$ telescopes, $\bm s_i=-[G_{i+1/2}-G_{i-1/2}]/h$ with the fixed defect fluxes $G_{i+1/2}=\tfrac{\Delta t^2}{2}\hat J_{i+1/2}(\bm w)$, where $\hat J$ is the total numerical flux of \cref{sec:setting}, and $G$ vanishes at the boundary interfaces by the zero-flux closure, so $\bm1^\top\bm s=0$.

The limited iteration of a stage combines the iterated antidiffusive fluxes and the fixed defect fluxes into one total correction flux per interface and clamps the sum,
\begin{equation}\label{eq:combflux}
F_{i+1/2}\bigl(\bm p\bigr)\;=\;\Delta t\,d_{i+1/2}(\bm p)\;+\;G_{i+1/2},
\qquad
\widetilde\Lambda_{i+1/2}(\bm p)\;=\;\max\Bigl(-\tfrac{h\,p^{\,n}_{i+1}}{2},\,\min\bigl(F_{i+1/2}(\bm p),\,\tfrac{h\,p^{\,n}_{i}}{2}\bigr)\Bigr),
\end{equation}
and the sweep solves $(\mathcal I-\Delta t\,A_1)\,\bm p^{[k+1]}=\bm p^{\,n}-\bigl[\widetilde\Lambda_{i+1/2}(\bm p^{[k]})-\widetilde\Lambda_{i-1/2}(\bm p^{[k]})\bigr]/h$, with budgets drawn from the common nonnegative base $\bm p^{\,n}$ in both stages. Clamping the sum rather than each contribution separately is deliberate: the two flux families may partially cancel, and the combined clamp discards only what the joint budget cannot afford. The predictor is the special case $G\equiv0$, which recovers \eqref{eq:picardlim} with the recomputed limiter.

\begin{proposition}[FCDF-DC]\label{prop:dc}
Let $\bm p^{\,n}\ge0$, let $A_1$ be the M-matrix core of \eqref{eq:coresplit}, and let both stages of \eqref{eq:predictor}--\eqref{eq:corrector} be realized by the limited iteration with the combined clamp \eqref{eq:combflux}. Then, for every $\Delta t>0$:
\begin{enumerate}
\item[\textup{(i)}] (Unconditional positivity) Every iterate of both stages is entrywise nonnegative, and so are the stage fixed points $\bm Y_0^{*}$ and $\bm p^{\,n+1}$.
\item[\textup{(ii)}] (Exact conservation and $\ell_1$ stability) $\bm1^\top\bm p^{\,n+1}=\bm1^\top\bm p^{\,n}$ for every limiter value in both stages, hence $\norm{\bm p^{\,n+1}}_1=\norm{\bm p^{\,n}}_1$ on nonnegative data and the time stepping is unconditionally $\ell_1$-stable on the cone.
\item[\textup{(iii)}] (Contraction) Both sweep maps are Lipschitz in $\ell_1$ with the same constant $q=2\Delta t\,\bar\mu/h$ as in \cref{prop:schemeBp}, since the fixed fluxes $G$ cancel in differences and the clamp is nonexpansive. For $\Delta t<h/(2\bar\mu)$ each stage converges geometrically to its unique fixed point.
\item[\textup{(iv)}] (Conditional temporal order) On any region where the caps are slack at the fixed points of both stages, the computed step coincides with the linear scheme \eqref{eq:predictor}--\eqref{eq:corrector} row by row, whose stability function is \eqref{eq:Rfunc}, so the local temporal error there is $O(\Delta t^3)$ and the step is second-order in time and space. Where a cap binds, the step degrades locally toward the limited backward Euler solution, first-order in time, at exactly the nodes where the spatial limiter of \cref{prop:schemeBp} is also active.
\end{enumerate}
\end{proposition}

\begin{proof}[Proof sketch]
The key observation is that the combined clamp \eqref{eq:combflux} is the composition of a translation by the fixed defect flux $G$ with a clamp onto the same budget interval as before. The translation is an isometry, so the fixed fluxes cancel in differences and do not enlarge the clamp's range. Positivity and conservation then follow for both stages exactly as in \cref{prop:schemeBp}: the withdrawals per node are capped by half of the budget of the common base $\bm p^{\,n}$ regardless of $G$, and the clamped fluxes telescope for every limiter value, which also makes the step $\ell_1$-isometric on the cone. The contraction bound is unchanged for the same reason, since only the iterated antidiffusive part varies between sweeps. For the temporal order, wherever the caps are slack at the fixed points of both stages the clamps are inactive, the unclamped fluxes reassemble the full operators, and the two fixed-point systems coincide row by row with the linear scheme \eqref{eq:predictor}--\eqref{eq:corrector}, whose stability function \eqref{eq:Rfunc} matches $e^z$ through $z^2$. A binding cap replaces the combined flux by a budget value, which is the local fallback toward limited backward Euler, and the binding condition requires the flux to exceed half the local density scale, which confines it to the same set as the spatial limiter. The full proof is given in \cref{appDC}.
\end{proof}

The temporal error therefore inherits the taxonomy of \cref{sec:peclet}: second order uniformly on resolved regions, the monotone floor inside unresolved layers, and the two mechanisms activate at the same interfaces. Two practical remarks close the section. First, the defect fluxes are computed from the limited predictor $\bm Y_0^{*}$, and \cref{lem:defect} is exact only where the predictor caps were slack, so a binding predictor cap perturbs the corrector's defect at the same layer interfaces already counted in \cref{prop:dc}(iv), not elsewhere. Second, the division of labor across the step-size axis is now complete on the temporal side as well: FCDF-DC covers $\Delta t<\gamma_{\mathrm{pic}}$ and Pad\'e$(0,2)$ covers $\Delta t\ge\gamma_r$, so the temporal analogue of the overlap condition \eqref{eq:overlap} is condition \eqref{eq:overlap2} of \cref{cor:coverage}. \Cref{ssec:numOverlap} reports $\gamma_r$ alongside $\gamma_0$ and $\gamma_{\mathrm{pic}}$ on a mesh sequence and finds this temporal overlap empty, which is what sends the large-step second-order branch to the active-set solver.

\section{Coverage of all step sizes} \label{sec:coverage}

This section synthesizes the small-step nonlinear schemes with the large-step linear positivity windows to establish complete coverage of the step-size parameter $\gamma > 0$. We formulate explicit overlap conditions that guarantee positivity and conservation across all step sizes, and introduce an active-set solver for regimes where these conditions fail.

\subsection{Hand-over to the linear windows} \label{ssec:handover}

The companion paper \cite{ItkinKazbek2026ADI} establishes two linear positivity windows for the irreducible, conservative, eventually exponentially positive operator $A_2$. The backward Euler resolvent satisfies $(\mathcal I-\gamma A_2)^{-1}>0$ for all $\gamma\ge\gamma_0$. The Pad\'e$(0,2)$ map satisfies $r_{02}(\gamma A_2)=(\mathcal I-\gamma A_2+\tfrac{\gamma^2}{2}A_2^2)^{-1}>0$ for all $\gamma\ge\gamma_r$. Both maps lose positivity for small $\gamma$, both thresholds are computable global properties of the assembled matrix, and neither dominates the other a priori. The two linear branches differ in temporal accuracy: the resolvent is first-order in time, the Pad\'e map is second-order and L-stable.

The nonlinear schemes of \cref{sec:limited,sec:timeorder} occupy the complementary regime. FCDF-B and its defect-corrected extension FCDF-DC are unconditionally positive and exactly conservative for every step, and their Picard iterations are contractive for $\gamma<\gamma_{\mathrm{pic}}=h/(2\bar\mu)$, the purely convective bound of \cref{prop:schemeBp}(iii). FCDF-B realizes the backward Euler step, FCDF-DC the second-order two-stage step. The four branches therefore pair off by temporal order, and coverage of the step-size axis reduces to two computable comparisons.

\begin{corollary}[Coverage]\label{cor:coverage}
Let $\gamma_{\mathrm{pic}}=h/(2\bar\mu)$ and let $\gamma_0$, $\gamma_r$ be the linear thresholds above.
\begin{enumerate}
\item[\textup{(a)}] (Positivity and conservation) If
\begin{equation}\label{eq:overlap}
\gamma_{\mathrm{pic}} \;\ge\; \min(\gamma_0,\gamma_r),
\end{equation}
then for every $\gamma>0$ the directional step \eqref{eq:implicitstep} admits a realization that is entrywise nonnegative on nonnegative data, exactly mass-conservative, and second-order consistent in space wherever the density is resolved: FCDF-B for $\gamma\le\gamma_{\mathrm{pic}}$ and whichever linear map is positive for $\gamma\ge\min(\gamma_0,\gamma_r)$, with an explicit switchover anywhere in the overlap.
\item[\textup{(b)}] (Full second order) If, in addition,
\begin{equation}\label{eq:overlap2}
\gamma_{\mathrm{pic}} \;\ge\; \gamma_r,
\end{equation}
then the realization can be chosen second-order in time as well: FCDF-DC for $\gamma\le\gamma_{\mathrm{pic}}$ and the Pad\'e$(0,2)$ map for $\gamma\ge\gamma_r$, so the time stepping is positive, conservative, and second-order in time and space, wherever the density is resolved, for every step size.
\end{enumerate}
Both conditions are single computable numbers per mesh, exact for the OU benchmark by eigendecomposition, and \eqref{eq:overlap2} implies \eqref{eq:overlap}.
\end{corollary}

The switchover deserves one remark. Inside an overlap interval both branches are admissible and the choice is free. The two solutions differ by the local truncation error, so switching within a run introduces no order-reducing transient. In practice the natural rule is to switch at a fixed fraction of $\gamma_{\mathrm{pic}}$, where the Picard iteration still converges in a few sweeps, and the numerical section reports iteration counts across the overlap to calibrate that fraction.

A second remark concerns what the two conditions deliver in practice. The threshold estimates of \cite{ItkinKazbek2026ADI} bound $\gamma_0$ and $\gamma_r$ by the ratio of the reduced resolvent norm to the smallest entry of the discrete stationary density. That entry lives in the deepest tail of the domain and is exponentially small when the domain is wide. The two conditions therefore behave very differently under refinement, and the overlap check of \cref{ssec:numOverlap} measures both on the OU benchmark. Condition \eqref{eq:overlap} holds on every mesh from $n=101$ upward, because $\gamma_0$ falls with refinement faster than $\gamma_{\mathrm{pic}}$ does. Condition \eqref{eq:overlap2} fails on every mesh we tested, because $\gamma_r$ is governed by the deep-tail stationary mass and grows without bound as the domain resolves that tail. On the configurations examined here, positivity and conservation at all step sizes are delivered by the hand-over of part (a), while full second order at large steps is delivered by the active-set solver of \cref{ssec:activeset} rather than by the Pad\'e window. We state part (b) as a conditional result because its hypothesis is a computable property of the assembled matrix and may hold on narrower domains. We do not claim that it holds generically.

Table~\ref{tab:branches} summarizes the branches, their guarantees, and their costs.

\begin{table}[!htb]
\centering
\scalebox{0.9}{
\begin{tabular}{lllll}
\hline
Branch & Realization & Positivity & Order in time & Regime \\
\hline
FCDF-B  & limited Picard, backward Euler & unconditional & 1 & $\gamma<\gamma_{\mathrm{pic}}$ \\
FCDF-DC & two limited stages \eqref{eq:predictor}--\eqref{eq:corrector} & unconditional & 2 where resolved & $\gamma<\gamma_{\mathrm{pic}}$ \\
Resolvent & one banded solve & linear, $\gamma\ge\gamma_0$ & 1 & linear window \\
Pad\'e$(0,2)$ & one banded solve (conjugate pair) & linear, $\gamma\ge\gamma_r$ & 2 & linear window \\
Active set & semismooth Newton, \cref{ssec:activeset} & automatic at the solution & 1 or 2, stage-wise & any $\gamma$ \\
\hline
\end{tabular}
}
\caption{The branches of \cref{cor:coverage}. Every realization conserves mass exactly under the conservative closure. Costs per step are $O(n)$ times the number of Picard sweeps, one for the linear branches, and $O(n)$ times the number of pattern updates for the active set.}
\label{tab:branches}
\end{table}

\subsection{An active-set solver for large steps} \label{ssec:activeset}

The measurements of \cref{ssec:numOverlap} make this subsection load-bearing rather than optional. Condition \eqref{eq:overlap2} fails on the meshes tested, so the second-order branch at large steps is delivered here and not by the Pad\'e window. The same construction removes the step restriction of the Picard iteration whenever \eqref{eq:overlap} fails, or whenever the FCDF branch is preferred beyond $\gamma_{\mathrm{pic}}$ for other reasons.

The fixed points of the sweep map of \cref{prop:schemeBp} are exactly the zeros of the residual
\begin{equation}\label{eq:residual}
\bm F(\bm p) \;:=\; \bigl(\mathcal I-\gamma A_1\bigr)\bm p \;-\; \gamma\,D_\Lambda(\bm p) \;-\; \bm b ,
\end{equation}
where $(D_\Lambda(\bm p))_i = -[\Lambda_{i+1/2}(\bm p)-\Lambda_{i-1/2}(\bm p)]/h$ is the divergence of the clamped fluxes. At each cell face $i+1/2$ (an interface) the flux $\Lambda_{i+1/2}$ is obtained by clamping (i.e., limiting via the median) the linear flux $d_{i+1/2}(\bm p)$ to the two caps (the lower and upper bounds from the limiter budget): it is the median of the lower cap, the upper cap, and $d_{i+1/2}(\bm p)$. Consequently, $\bm F$ is piecewise linear, globally Lipschitz, and semismooth. This formalizes the limited system exactly; the limiter is driven by the caps, not by the sign of the solution, and \eqref{eq:residual} replaces the informal description as a complementarity problem in $(\bm p,\theta)$.

The semismooth Newton method for $\bm F(\bm p)=\bm 0$ has a transparent structure. At the current iterate each interface $i+1/2$ is labeled by its clamp pattern: lower cap, free, or upper cap. On a free interface the clamping operation reduces to the identity, so the flux equals the linear flux $d_{i+1/2}(\bm p)$ and contributes the correction flux. On a clamped interface the flux is fixed to a cap and contributes nothing. For a given pattern $S$, the generalized Jacobian is therefore
\begin{equation}\label{eq:newtonmat}
V_S \;=\; \mathcal I-\gamma\bigl(A_1+C_S\bigr),
\end{equation}
where $C_S$ is the correction operator \eqref{eq:corrflux1} restricted to the free interfaces of $S$. The matrix $A_1+C_S$ interpolates between the monotone core (all interfaces clamped) and the full second-order operator $A_2$ (all interfaces free). For every pattern $V_S$ is banded with the bandwidth of $A_2$. Each Newton step solves one such banded system and updates the pattern. Because $\bm F$ is piecewise linear, the iteration terminates as soon as the pattern at the solution is identified; one further solve then lands exactly on the zero. Because $\bm F$ is piecewise linear, the iteration terminates as soon as the pattern at the solution is identified; one further solve then lands exactly on the zero. The iteration count is therefore the number of pattern updates rather than a rate of convergence, and every solve reuses the same banded machinery. \Cref{ssec:numActiveSet} measures that count. It does not grow with $\gamma$, and it does not grow with the mesh.

Three properties support this construction. First, below the convective bound all Newton matrices are uniformly nonsingular: $V_S=M(\mathcal I-\gamma M^{-1}C_S)$ and $\norm{\gamma M^{-1}C_S}_1\le2\gamma\bar\mu/h<1$ for every pattern $S$, by the unit norm of the core resolvent and the flux bound of \cref{prop:schemeBp}(iii). Second, inside the resolvent window the all-free matrix $V_{\mathrm{free}}=\mathcal I-\gamma A_2$ is inverse-positive by \cite{ItkinKazbek2026ADI}. Hence the algorithm's first move, solving the unlimited system and accepting a nonnegative result at zero pattern updates and full second order, is the generic outcome there. Third, any zero of $\bm F$ is a fixed point of the sweep map and therefore nonnegative by \cref{prop:schemeBp}(i). What remains outside these statements is the uniqueness of zeros for $\gamma$ between $\gamma_{\mathrm{pic}}$ and the linear window. The iteration is initialized from the unlimited solution, which is available from the first move at no extra cost, and \cref{ssec:numActiveSet} reports pattern-update counts across $\gamma$, across meshes, and on data designed to fragment the active set.

\begin{proposition}[Nonsingularity of mixed-pattern Newton matrices]\label{prop:nonsingular}
Let $M=\mathcal I-\gamma A_1$ be the monotone core resolvent. For $\gamma\le\gamma_{\mathrm{pic}}$, $M$ is nonsingular and satisfies $\norm{M^{-1}}_1\le 1$. For any pattern $S$, factor
\begin{equation}\label{eq:factor}
V_S = M\bigl(\mathcal I - \gamma M^{-1}C_S\bigr).
\end{equation}
Then the following hold:
\begin{enumerate}
    \item[(a)] For $\gamma\le\gamma_{\mathrm{pic}}$, all matrices $V_S$ are uniformly nonsingular.
    \item[(b)] For $\gamma\in(\gamma_{\mathrm{pic}},\gamma_r)$, any pattern $S$ whose clamped interfaces form a set in which each contiguous block is separated from the next by at least one free interface yields a nonsingular $V_S$.
    \item[(c)] For $\gamma\le\gamma_r$, the all-free matrix $V_{\mathrm{free}}=\mathcal I-\gamma A_2$ and the all-clamped matrix $M$ are nonsingular. Nonsingularity of $V_S$ for arbitrary mixed patterns when $\gamma>\gamma_{\mathrm{pic}}$ remains an open theoretical question; numerical experiments to date show no failure across the tested range of $\gamma$.
\end{enumerate}
\end{proposition}

\begin{proof}
Part (a) follows from the flux bound of \Cref{prop:schemeBp}(iii): $\norm{\gamma M^{-1}C_S}_1\le 2\gamma\bar\mu/h<1$ uniformly in $S$ for $\gamma\le\gamma_{\mathrm{pic}}$. Hence $\mathcal I-\gamma M^{-1}C_S$ is a small perturbation of the identity and nonsingular by Neumann series, and $V_S$ inherits nonsingularity from $M$.

Part (c) is immediate: $V_{\mathrm{free}}$ is inverse-positive inside the resolvent window ($\gamma\le\gamma_r$) by \cite{ItkinKazbek2026ADI}, and $M$ is nonsingular by construction for $\gamma\le\gamma_{\mathrm{pic}}$, hence also for $\gamma\le\gamma_r$ when $\gamma_{\mathrm{pic}}<\gamma_r$.

Part (b) requires a more detailed argument, given as \Cref{lem:separated} below.
\end{proof}

\begin{lemma}[Separated clamped interfaces]\label{lem:separated}
Let $S$ be a pattern whose set of clamped interfaces can be partitioned into contiguous blocks $B_1,\dots,B_k$, and assume that any two distinct blocks $B_i$, $B_j$ are separated by at least one free interface. For $\gamma\in(\gamma_{\mathrm{pic}},\gamma_r)$, the matrix $V_S$ is nonsingular.
\end{lemma}

\begin{proof}
See \cref{app:separated}
\end{proof}

The construction extends to the temporal branch without modification. Both stages of FCDF-DC are solves of the same form: a banded system with a nonnegative base and flux-form sources. Hence the residual \eqref{eq:residual} with the combined clamp of \eqref{eq:combflux} covers both stages, and the active-set solver applies stage-wise. The second-order-in-time branch of \cref{cor:coverage}(b) is thereby available beyond $\gamma_{\mathrm{pic}}$ as well.

Two gaps are closed by this solver, and they are of different kinds. The first is the temporal gap of \eqref{eq:overlap2}, which \cref{ssec:numOverlap} finds open on every mesh of the OU benchmark. The second is a gap in \eqref{eq:overlap} itself, which the smooth benchmark does not exhibit but the advection-dominated one does. On that problem $\gamma_0$ sits about four times above $\gamma_{\mathrm{pic}}$ on every mesh, so there is an interval of step sizes in which the Picard iteration no longer contracts and the resolvent is not yet positive, and the interval does not close under refinement. \Cref{ssec:numActiveSet} reports that a step in the middle of it costs one pattern update.

\section{Numerical experiments} \label{sec:numerics}

All experiments use the one-dimensional schemes of this paper on a uniform grid with the conservative zero-flux closure of \cref{sec:setting}. Five schemes are compared. FCDF-B is the flux-limited iteration of \cref{prop:schemeBp} with the limiter recomputed at every sweep. FCDF-DC is the defect-corrected step of \cref{prop:dc}. The unlimited linear scheme is backward Euler on the full second-order operator $A_2$, included to show where positivity fails. The monotone core is backward Euler on $A_1$, the first-order positive floor. CC is the exponentially fitted Chang--Cooper scheme, assembled interface by interface so that it remains valid for the variable and sign-changing drift of the Ornstein--Uhlenbeck benchmark.

Errors are reported in the discrete $L_1$ norm $\norm{\bm v}_{1,h}=h\sum_i|v_i|$ against the analytic solution wherever one exists. Observed orders are the successive-halving values, which expose pre-asymptotic behaviour that a single least-squares slope conceals. Mass is monitored through the discrete invariant $\bm1^\top\bm p$, which the conservative closure preserves exactly, and positivity through the most negative nodal value.

One methodological point governs every convergence table below. A study at fixed $\Delta t$ measures the spatial order only while the temporal error stays below the spatial one. Every scheme here that is first order in time, which is FCDF-B, the unlimited scheme, the monotone core and CC, therefore shows its measured spatial order fall away on the finest meshes, and that fall is a property of the measurement rather than of the stencil. FCDF-DC, being second order in time, is free of this effect. Where a statement about the stencil alone is needed, we compare the semi-discrete solution $e^{TA}\bm p^{\,0}$ with the analytic density, which removes the time discretization entirely.

\subsection{The overlap check} \label{ssec:numOverlap}

The first experiment decides the two conditions of \cref{cor:coverage}. The benchmark is the Ornstein--Uhlenbeck process $X_t$, governed by the stochastic differential equation
\begin{equation}
dx_t = -\alpha x_t \, dt + \sigma \, dW_t,
\end{equation}
where $W_t$ is a standard Brownian motion and the parameters are set to $\alpha=\sigma=1$. The domain spans six stationary standard deviations on either side of the origin. This ensures that the drift term, $\mu(x)=-\alpha x$, changes sign at the centre of the domain.

For each mesh the operator $A_2$ is assembled and the three thresholds are computed. The first threshold, $\gamma_{\mathrm{pic}}=h/(2\bar\mu)$, is calculated in closed form. The remaining two, $\gamma_0$ and $\gamma_r$, are determined by bisection on the most negative entry of the corresponding resolvent map.

\begin{table}[!htb]
\centering
\begin{tabular}{rcccccc}
\hline
$n$ & $h$ & $\gamma_{\mathrm{pic}}$ & $\gamma_0$ & $\gamma_r$ & \eqref{eq:overlap} & \eqref{eq:overlap2} \\
\hline
$51$ & $1.70\times10^{-1}$ & $2.00\times10^{-2}$ & $4.28\times10^{-2}$ & $1.36\times10^{0}$ & no & no \\
$101$ & $8.49\times10^{-2}$ & $10^{-2}$ & $9.84\times10^{-3}$ & $1.39\times10^{0}$ & yes & no \\
$201$ & $4.24\times10^{-2}$ & $5.00\times10^{-3}$ & $1.82\times10^{-3}$ & $\infty$ & yes & no \\
$401$ & $2.12\times10^{-2}$ & $2.50\times10^{-3}$ & $2.69\times10^{-4}$ & $3.80\times10^{5}$ & yes & no \\
\hline
\end{tabular}
\caption{Coverage thresholds on the OU benchmark. The entry $\infty$ means that no positive window was found below the search bound $10^{6}$. The decay observed in the first-order-in-time schemes (FCDF-B, unlimited, CC) on fine meshes reflects the temporal error floor, not a degradation of their spatial stencils."}
\label{tab:overlap}
\end{table}

The two conditions behave in opposite ways under refinement. The resolvent threshold $\gamma_0$ falls faster than $\gamma_{\mathrm{pic}}$, so condition \eqref{eq:overlap} fails only on the coarsest mesh and holds from $n=101$ onward, with the margin widening as the mesh refines. The Pad\'e threshold $\gamma_r$ does the reverse. It is already two orders of magnitude above $\gamma_{\mathrm{pic}}$ on the coarsest mesh and grows without bound, so condition \eqref{eq:overlap2} fails on every mesh tested. The mechanism is the one anticipated in \cref{ssec:handover}. The threshold estimates are governed by the smallest entry of the discrete stationary density, that entry sits in the deepest tail of a domain six standard deviations wide, and refining the mesh resolves the tail further and drives the entry towards zero.

The implications for the proposed method are straightforward. For this benchmark, positivity and exact conservation are guaranteed at every step size by \cref{cor:coverage}(a). However, full second-order accuracy at large step sizes is not achieved via the Pad'e window alone. Rather, it relies on the active-set solver detailed in \cref{ssec:activeset}.

\Cref{fig:thresholds} compares the three key theoretical thresholds, the Picard contraction threshold~$\gamma_{\mathrm{pic}}$, the backward Euler resolvent positivity threshold~$\gamma_0$, and the Pad\'{e}(0,2) positivity threshold~$\gamma_r$, as functions of the mesh size~$h$ on the Ornstein--Uhlenbeck benchmark, determining whether the coverage conditions of Corollary~\ref{cor:coverage} hold.

\begin{figure}[!htb]
\centering
\includegraphics[width=0.6\textwidth]{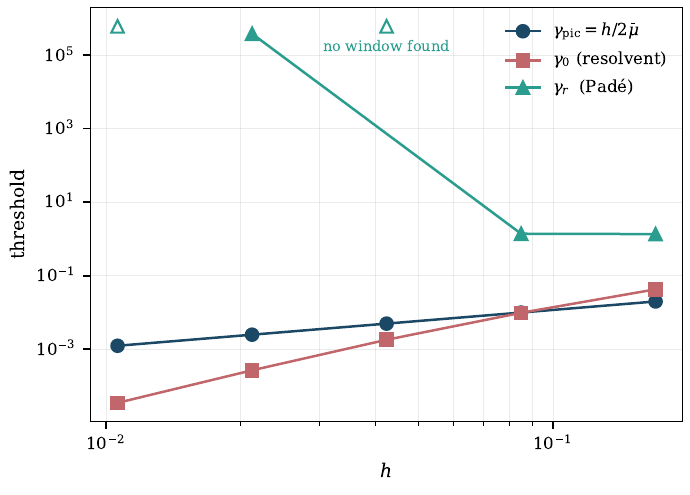}
\caption{Coverage thresholds against the mesh on the OU benchmark. The Picard threshold falls linearly in $h$. The resolvent threshold falls faster and crosses below it, so condition \eqref{eq:overlap} holds from $n=101$ onward. The Pad\'e threshold stays three orders of magnitude above and escapes upward, so condition \eqref{eq:overlap2} fails on every mesh. Open symbols mark meshes on which no Pad\'e window was found below the search bound $10^{6}$.}
\label{fig:thresholds}
\end{figure}

\subsection{Ornstein--Uhlenbeck: accuracy against the analytic solution} \label{ssec:numOU}

Given an initial Gaussian distribution with mean $x_0$ and variance $v_0$, the exact OU density remains Gaussian with mean $x_0 e^{-\alpha t}$ and variance $(\sigma^2/2\alpha)(1-e^{-2\alpha t})+v_0 e^{-2\alpha t}$, so the error is measured against an exact target. We take initial values $x_0=0.5$ and $v_0=10^{-2}$ and integrate to $T=0.2$. The maximum cell P\'eclet number on these meshes is $0.36$, meaning the problem is diffusion-dominated and the limiter is expected to remain inactive except in the tails.

\Cref{tab:ouspace} fixes $\Delta t=2\times10^{-4}$ and refines the mesh. FCDF-DC converges at second order throughout, reaching $1.98$ on the finest pair. FCDF-B and the unlimited scheme produce errors that agree to every digit reported (but the solutions are not identical), which is \cref{prop:schemeBp}(iv) holding exactly: on a smooth density the budgets are slack at every interface, the limiters equal one, and the limited fixed point is the unlimited solution. Their common order falls on the last two meshes because the fixed time step no longer isolates the spatial error, and the same happens to CC. The monotone core sits at first order as expected.

\begin{table}[!htb]
\centering
\begin{tabular}{cccccc}
\hline
$h$ & FCDF-DC & FCDF-B & unlimited & CC & core \\
\hline
$1.697\times10^{-1}$ & --   & --   & --   & --   & --   \\
$8.485\times10^{-2}$ & 1.80 & 1.78 & 1.78 & 2.03 & 0.74 \\
$4.243\times10^{-2}$ & 1.90 & 1.80 & 1.80 & 1.73 & 0.93 \\
$2.121\times10^{-2}$ & 1.96 & 1.59 & 1.59 & 1.24 & 0.98 \\
$1.061\times10^{-2}$ & 1.98 & 0.93 & 0.93 & 0.60 & 1.04 \\
\hline
\end{tabular}
\caption{Observed $L_1$ orders in space on the OU benchmark at $\Delta t=2\times10^{-4}$, $T=0.2$. The decay of the last two columns on fine meshes is the temporal floor of the first-order-in-time schemes, not a property of their stencils.}
\label{tab:ouspace}
\end{table}

In \Cref{tab:outime}, the spatial grid is fixed at $n=401$ while the time step is successively refined. Each scheme is compared against the exact solution of its respective semi-discrete system: $e^{TA_2}\bm p^{\,0}$ for FCDF-B and FCDF-DC, and $e^{TA_1}\bm p^{\,0}$ for the monotone core. Utilizing a single shared reference would unfairly penalize the monotone core by incorporating the fixed difference between the two spatial operators. On this particular mesh, this discrepancy is $5.3\times10^{-3}$, an error that no amount of temporal refinement can eliminate.

\begin{table}[!htb]
\centering
\begin{tabular}{ccccc}
\hline
$\Delta t$ & FCDF-DC error & order & FCDF-B order & core order \\
\hline
$8.0\times10^{-3}$ & $2.071\times10^{-3}$ & --   & --   & --   \\
$4.0\times10^{-3}$ & $4.668\times10^{-4}$ & 2.15 & 1.00 & 1.00 \\
$2.0\times10^{-3}$ & $4.821\times10^{-5}$ & 3.28 & 1.00 & 1.00 \\
$1.0\times10^{-3}$ & $4.994\times10^{-6}$ & 3.27 & 1.00 & 1.00 \\
$5.0\times10^{-4}$ & $1.317\times10^{-6}$ & 1.92 & 1.00 & 1.00 \\
$2.5\times10^{-4}$ & $3.320\times10^{-7}$ & 1.99 & 1.00 & 1.00 \\
\hline
\end{tabular}
\caption{Observed $L_1$ orders in time on the OU benchmark at $n=401$, $T=0.2$, against the exact-in-time solution of each scheme's own semi-discrete system.}
\label{tab:outime}
\end{table}

FCDF-B and the monotone core exhibit first-order convergence in time to three significant digits across all tested step sizes. FCDF-DC, however, passes through a pre-asymptotic phase where the observed order initially exceeds three. This occurs because the leading second-order error term carries a small coefficient for this problem, causing the subsequent term to dominate until the step size becomes sufficiently small, after which the rate settles at $1.92$ and $1.99$. Simultaneously refining space and time such that $\Delta t\sim h$ yields observed convergence rates of $1.80$, $1.88$, and $1.95$ for FCDF-DC. This confirms the statement of \cref{cor:coverage}(b) in the regime where the caps are slack.

\subsection{P\'eclet sweep: FCDF versus Chang--Cooper} \label{ssec:numPeclet}

The contrast of \cref{sec:peclet} is between a stencil that is modified everywhere according to the cell P\'eclet number and a limiter that responds to the solution. It is therefore a statement about spatial discretization, and it is measured on a smooth solution, where the FCDF budgets are never exceeded. The benchmark is constant-coefficient advection diffusion, for which a Gaussian initial density stays Gaussian with mean $x_0+\mu t$ and variance $v_0+2Dt$. The domain is padded by eight standard deviations of the final profile on both sides, so the zero-flux boundaries are never reached and the free-space solution is an exact reference. The P\'eclet number is swept by lowering $D$ at fixed mesh sequence.

\Cref{tab:peclet} reports the pure spatial order, obtained by comparing $e^{TA}\bm p^{\,0}$ with the analytic density for the two linear operators involved. Both operators are exact in time under this comparison, so the two columns differ only through their stencils.

\begin{table}[!htb]
\centering
\begin{tabular}{ccccc}
\hline
$D$ & $\mathrm{Pe}_h$ (finest) & $A_2$ order & CC order & error ratio \\
\hline
$10^{-1}$ & $0.03$ & $1.97$ & $2.00$ & $0.5$ \\
$10^{-2}$ & $0.15$ & $1.99$ & $1.98$ & $1.1$ \\
$10^{-3}$ & $1.16$ & $1.99$ & $1.63$ & $8.1$ \\
$3.00\times10^{-4}$ & $3.79$ & $1.99$ & $1.21$ & $21.9$ \\
\hline
\end{tabular}
\caption{Pure spatial $L_1$ order on the smooth advection benchmark, from the semi-discrete solutions, at $T=0.1$ on the mesh sequence $n=101,\dots,801$. The last column is the ratio of the CC error to the $A_2$ error on the finest mesh.}
\label{tab:peclet}
\end{table}

The numerical evidence is unambiguous: the second-order upwind operator of the DF discretization maintains a convergence order of $1.99$ across the entire P\'eclet number sweep. The CC scheme coincides with this at low P\'eclet numbers where the fitting weight approaches one half and the stencil is nearly centred, but degrades monotonically as convection dominates. At $\mathrm{Pe}_h \approx 3.8$, the order drops to $1.21$, producing an absolute error twenty-two times larger. As predicted by the fitting construction, this degradation is unconditional and independent of the solution. Furthermore, evaluating the fully discrete schemes yields convergence rates of $1.97$, $1.98$, $1.99$, and $1.99$ for FCDF-DC across the same parameter sweep, confirming that the limited scheme achieves its theoretical stencil order in practice.

The second part of the study is the order-one front, where \cref{prop:peclet}(b) predicts the monotone floor. We take $\mu=1$, $D=10^{-4}$ and a plateau of unit mass on $[0.1,0.4]$, so the front width $D/\mu=10^{-4}$ is below every mesh size used and the front is unresolved throughout. FCDF-B converges at observed orders $0.95$, $1.18$ and $1.47$, and CC at $0.65$, $0.73$ and $0.85$, with FCDF-B carrying an absolute error five times smaller on the finest mesh. Both schemes remain nonnegative and conserve mass to round-off. Neither reaches second order, which is the floor rather than a defect, and the comparison of the two at a front says little about the mechanism, which is why the smooth study above carries the claim.

This run provides direct verification of the limiter's localization. Although the limiter is active on $38\%$ of the interfaces at $n=801$, this statistic is misleading in isolation. The maximum density adjacent to any limited interface is merely $9.9\times10^{-3}$ (against an overall peak density of $3.38$), and the total mass contained within the limited region is only $1.7\times10^{-5}$ out of a total mass of one. Consequently, the limiter activates strictly where the density has dropped to the truncation level. This behavior confirms the localization claimed in \cref{prop:schemeBp}(iv) and rigorously quantified by \eqref{eq:suffslack}. Ultimately, the relevant metric for limiter impact is not the raw interface count, but rather the density scale within the limited region.

\Cref{fig:limiter} visualizes the spatial localization of the flux limiter on a logarithmic density scale for the advection-dominated front benchmark with $D=10^{-4}$ at $n=801$ grid points, confirming that the limiter activates exclusively in regions where the density has fallen to the truncation-error level.

\begin{figure}[!htb]
\centering
\includegraphics[width=0.6\textwidth]{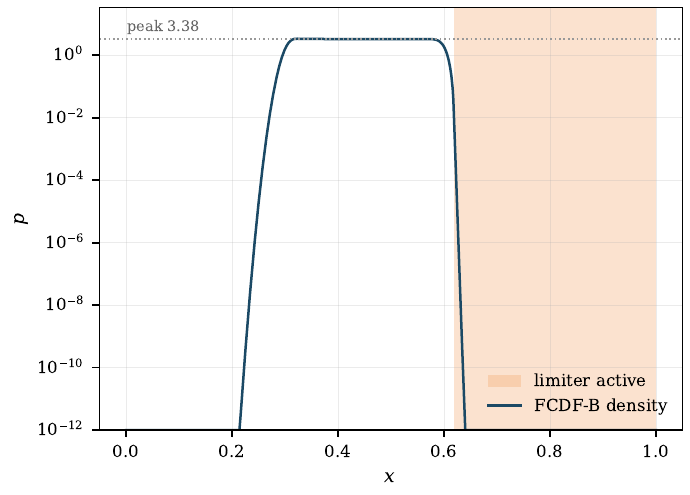}
\caption{Limiter activity at an unresolved front, $D=10^{-4}$ and $n=801$, on a logarithmic density scale. The shaded interfaces are those where the clamp is active. They begin only where the density has fallen several orders of magnitude below the plateau, and the advected packet itself is untouched. The clamp is active on $38\%$ of the interfaces, which carry $1.7\times10^{-5}$ of the total mass.}
\label{fig:limiter}
\end{figure}

\subsection{Long-time front smearing} \label{ssec:numLongtime}

\Cref{prop:peclet} bounds a single resolvent solve. This experiment asks what survives over many steps, when a monotone scheme smears a front progressively and the pure-advection rate of the monotone core is known to fall to $O(h^{1/2})$. The front benchmark is integrated at fixed $\Delta t=2\times10^{-4}$ to four horizons, and the spatial order is measured at each.

\begin{table}[!htb]
\centering
\begin{tabular}{ccccc}
\hline
$T$ & steps & FCDF-B order & core order \\
\hline
0.05 &  250 & 1.05 & 0.70 \\
0.10 &  500 & 1.12 & 0.69 \\
0.20 & 1000 & 1.18 & 0.68 \\
0.40 & 2000 & 1.25 & 0.67 \\
\hline
\end{tabular}
\caption{Observed spatial $L_1$ order at a front as the number of steps grows, on the advection benchmark with $D=10^{-4}$.}
\label{tab:longtime}
\end{table}

The two schemes separate and stay separated. FCDF-B holds an observed rate above one at every horizon, and the rate does not deteriorate as the step count grows by a factor of eight. The monotone core settles near two thirds and drifts slowly downward, consistent with the $O(h^{1/2})$ limit of \cite{Kuznetsov1976,Sabac1997}. At the longest horizon the absolute errors are $1.37\times10^{-2}$ for FCDF-B and $1.08\times10^{-1}$ for the core, a factor of eight. The single-solve rate of \cref{prop:peclet}(b) therefore survives long integration, which is the question left open after that proposition.

\subsection{Positivity and conservation} \label{ssec:numAudit}

The last experiment is not an accuracy study. It reports the structural guarantees across the benchmark family, on a smooth problem and a front problem, at a small and a large time step.

\begin{table}[!htb]
\centering
\begin{tabular}{l S S S S}
\hline
{scheme} & {OU, $\Delta t$ small} & {OU, $\Delta t$ large} & {front, $\Delta t$ small} & {front, $\Delta t$ large} \\
\hline
FCDF-B     & 9e-27 & 2e-14 & 2e-133   & 8e-60 \\
FCDF-DC    & 3e-27 & 2e-14 & 2e-77    & 1e-58 \\
unlimited  & 9e-27 & 8e-15 & -2.55e-1 & 1e-65 \\
core       & 4e-26 & 2e-14 & 2e-62    & 8e-59 \\
CC         & 4e-27 & 7e-15 & 0        & 0 \\
\hline
\end{tabular}
\caption{Most negative nodal value after integration to $T=0.2$. Small and large steps are $\Delta t=2\times10^{-4}$ and $\Delta t=5\times10^{-2}$. The mass defect $|\bm1^\top\bm p-\bm1^\top\bm p^{\,0}|$ never exceeds $2.2\times10^{-11}$ for any entry in the table.}
\label{tab:audit}
\end{table}

Of the tested methods, only the unlimited scheme produces a negative density, failing precisely as the theory predicts. It stays nonnegative on the smooth problem at both steps, because smooth positive data does not activate the wrong-signed band even when the resolvent matrix is not monotone. It fails on the front at the small step, where it undershoots to $-0.255$. It recovers at the large step, where the resolvent window of \cite{ItkinKazbek2026ADI} has been entered. Positivity is lost as $\Delta t \to 0$ and regained above the threshold, which is the mirror of the eventual-positivity statement and the reason the small-step regime needed a nonlinear scheme at all. An experiment run only at large steps would have shown no failure.

The FCDF schemes are nonnegative in every column, as are the monotone core and CC, and every scheme conserves the discrete mass to round-off, the largest defect over the whole table being $2.2\times10^{-11}$.

\subsection{The active-set solver} \label{ssec:numActiveSet}

The Picard iteration contracts only below $\gamma_{\mathrm{pic}}$, and \cref{ssec:numOverlap} showed that the Pad\'e window does not open in time to take over. The practical claim of \cref{ssec:activeset} is that the active-set solver covers the remainder cheaply, and that its cost is set by the clamp pattern rather than by the step size. This subsection measures that.

\Cref{tab:asgamma} sweeps the step over six decades around $\gamma_{\mathrm{pic}}$ on the advection-dominated front problem. The solver is initialized from the unlimited solution in every row.

\begin{table}[!htb]
\centering
\begin{tabular}{rccccc}
\hline
$\gamma/\gamma_{\mathrm{pic}}$ & updates & unlimited accepted & free fraction & residual & mass defect \\
\hline
$0.01$    & 1 & no  & 0.300 & $3.6\times10^{-15}$ & $0$ \\
$0.1$     & 1 & no  & 0.300 & $4.9\times10^{-14}$ & $0$ \\
$0.5$     & 1 & no  & 0.300 & $5.7\times10^{-14}$ & $0$ \\
$1$       & 1 & no  & 0.300 & $5.7\times10^{-14}$ & $5.7\times10^{-14}$ \\
$2$       & 1 & no  & 0.300 & $2.8\times10^{-14}$ & $0$ \\
$5$       & 0 & yes & 1.000 & $2.7\times10^{-13}$ & $5.7\times10^{-14}$ \\
$10$      & 0 & yes & 1.000 & $6.3\times10^{-13}$ & $8.5\times10^{-13}$ \\
$10^{2}$  & 0 & yes & 1.000 & $3.6\times10^{-12}$ & $2.4\times10^{-12}$ \\
$10^{3}$  & 0 & yes & 1.000 & $1.5\times10^{-11}$ & $3.8\times10^{-12}$ \\
$10^{4}$  & 0 & yes & 1.000 & $3.4\times10^{-11}$ & $1.7\times10^{-13}$ \\
\hline
\end{tabular}
\caption{Pattern updates against step size on the front problem at $n=401$, where $\gamma_{\mathrm{pic}}=1.25\times10^{-3}$. The residual is $\norm{\bm F(\bm p)}_1$ where the limited system is solved, and $\norm{(\mathcal I-\gamma A_2)\bm p-\bm b}_1$ where the unlimited solution is accepted.}
\label{tab:asgamma}
\end{table}

The cost falls as the step grows, which is the opposite of the Picard behaviour and the reason the two mechanisms compose. Up to twice the Picard threshold one pattern update suffices, and the clamp is active on seventy per cent of the interfaces. From five times the threshold onward the unlimited solve is already nonnegative and is accepted outright, at zero updates and full second order, because the resolvent window of \cite{ItkinKazbek2026ADI} has been entered. The solver is therefore doing work only in the band where work is needed.

The count is also insensitive to the mesh. Repeating three of these step sizes over a refinement sequence gives one update at $0.5\gamma_{\mathrm{pic}}$ and at $2\gamma_{\mathrm{pic}}$, and zero at $100\gamma_{\mathrm{pic}}$, for every $n$ from $101$ to $1601$. Since each update is one banded solve at $O(n)$ cost, the total work per step is a small constant multiple of a single second-order solve.

\Cref{tab:asgap} isolates the case the solver exists for. On this problem condition \eqref{eq:overlap} fails, and it fails by a margin that refinement does not remove.

\begin{table}[!htb]
\centering
\begin{tabular}{rcccc}
\hline
$n$ & $\gamma_{\mathrm{pic}}$ & $\gamma_0$ & $\gamma_0/\gamma_{\mathrm{pic}}$ & updates inside the gap \\
\hline
$101$ & $5.00\times10^{-3}$ & $1.93\times10^{-2}$ & $3.86$ & 1 \\
$201$ & $2.50\times10^{-3}$ & $9.56\times10^{-3}$ & $3.82$ & 1 \\
$401$ & $1.25\times10^{-3}$ & $4.68\times10^{-3}$ & $3.75$ & 1 \\
$801$ & $6.25\times10^{-4}$ & $2.25\times10^{-3}$ & $3.60$ & 1 \\
\hline
\end{tabular}
\caption{The coverage gap on the front problem with $D=10^{-4}$. The last column is the number of pattern updates at $\gamma=\sqrt{\gamma_{\mathrm{pic}}\gamma_0}$, a step in the middle of the interval where neither the contraction argument nor the linear window applies.}
\label{tab:asgap}
\end{table}

The ratio stays near four on every mesh, so there is a genuine interval of step sizes that neither mechanism of \cref{cor:coverage}(a) reaches, and the interval does not close as $h\to0$. This is the complement of the OU picture in \cref{tab:overlap}, where condition \eqref{eq:overlap} holds and the gap is empty. Taken together the two benchmarks show that neither the hand-over nor the active set is dispensable. A step drawn from the middle of the gap costs one pattern update on every mesh.

Only one setting we found requires more. Data with many scattered near-zero nodes, generated as $u^{k}$ for uniform $u$ and $k=2,4,8,16$, fragments the clamped set into many short blocks and needs two or three updates, at step sizes from $0.1\gamma_{\mathrm{pic}}$ to $10\gamma_{\mathrm{pic}}$. The residual reaches $10^{-14}$ in every case and the returned density is nonnegative, as \cref{prop:schemeBp}(i) guarantees for any zero of $\bm F$. We did not find a configuration in which the iteration failed to converge.

\section{Discussion and Conclusions} \label{sec:conclusion}

This paper turns the runtime sign monitoring of the DF solvers into a scheme with proved properties. The second-order directional operator is split into a monotone M-matrix core and an antidiffusive flux correction. And the implicit step is realized by Picard iteration with a Zalesak-type clamp applied per interface inside the banded solve. The resulting FCDF-B scheme is unconditionally positive and exactly mass-conservative for every time step and every limiter value, its sweep map contracts in $\ell_1$ under the purely convective restriction $\gamma < \gamma_{\mathrm{pic}} = h/(2\bar\mu)$ with the explicit constant of \cref{prop:schemeBp}, and its fixed point solves the full second-order system at every node where the budgets are slack. None of this contradicts the barriers of Godunov and of Bolley and Crouzeix. The scheme is nonlinear precisely where the barriers demand it, and the price is paid locally, at the interfaces where the density sits at truncation level.

The accuracy statement is deliberately two-sided. Away from unresolved layers the scheme coincides with the unlimited second-order discretization, and the $L_1$ error obeys the P\'eclet-uniform two-term bound of \cref{prop:peclet}. Tail layers cost $O(h^2)$ and leave second order intact, which is the generic case for smooth Fokker--Planck densities and the contrast with the global first-order degradation of the CC scheme. Interior unresolved fronts of order-one height cost $O(h)$, the monotone floor, and sub-cell boundary accumulation falls outside nodal comparison altogether and is assessed by cellwise mass. The distinction from exponential fitting is one of locality rather than of magnitude at a single interface, and it is visible on smooth solutions at large P\'eclet number, where the limiter stays inactive and the order is retained while the fitted stencil degrades everywhere. We regard stating these limits as part of the result rather than a qualification of it.

On the temporal side, the defect-corrected scheme FCDF-DC restores second order in time from limited backward Euler solves alone, with the defect available in closed form and clamped jointly with the antidiffusive fluxes. Its positivity and exact conservation are unconditional, while second order in time holds where the caps are slack, degrading to the backward Euler floor on the same interfaces where the spatial limiter is active. Combined with the two linear windows of the companion paper, the resolvent above $\gamma_0$ and the Pad\'e$(0,2)$ map above $\gamma_r$, this yields the coverage corollary: positivity and conservation for all step sizes under the computable condition $\gamma_{\mathrm{pic}}\ge\min(\gamma_0,\gamma_r)$, and full second order in time and space, wherever the density is resolved, under $\gamma_{\mathrm{pic}}\ge\gamma_r$. Both conditions are single numbers per mesh, and we report them measured rather than assumed. The first holds on the benchmark family. The second does not, because the Pad\'e threshold is set by the deep-tail stationary mass and grows as the mesh resolves that tail. Hence, the second-order branch at large steps rests on the active-set solver, at a cost governed by the number of nodes where positivity binds.

Several questions remain open and are stated as such in the text. The localization claim of \cref{prop:schemeBp}(iv) rests on the comparability of the budget and the density, which we verify a posteriori rather than prove from the data, and a Harnack-type estimate for the core resolvent, \cite{Moser1964} would close it. The well-posedness of the limited system for all $\gamma$ and the local superlinear convergence of the semismooth Newton iteration are used in the form standard for complementarity problems and deserve a dedicated proof in this setting. On the modeling side, the backward equation admits two limited realizations, transpose-consistent or independently positive, and the choice is deferred to the companion papers together with the temporal order of the composite multidimensional step.

The schemes act line by line and are therefore directly usable as directional factors in the Strang composition of \cite{ItkinDF2026} and the ADI composition of \cite{ItkinKazbek2026ADI}, exactly in the regime those papers leave uncovered. Those are below their linear thresholds with the composite mass propositions holding unchanged by exact conservation for every limiter value. Together, the three papers assemble one picture: eventual positivity carries the exponential and rational maps above explicit thresholds, and the flux-corrected schemes of this paper carry the regime below them, so that positivity, exact conservation, and second order where the density is resolved are available at every step size for which one of the mechanisms applies.

\section*{Acknowledgments}

\noindent Numerical analysis and manuscript preparation were performed in part with the help of Opus 4.8 working as an AI assistant under author’s supervision.

\printbibliography[title={References}]

@Article{ChangCooper1970,
  author  = {Chang, J. S. and Cooper, G.},
  title   = {A practical difference scheme for Fokker-Planck equations},
  journal = {Journal of Computational Physics},
  year    = {1970},
  volume  = {6},
  number  = {1},
  pages   = {1--16},
}

@Article{Harten1983,
  author  = {Harten, Ami},
  title   = {High resolution schemes for hyperbolic conservation laws},
  journal = {Journal of Computational Physics},
  year    = {1983},
  volume  = {49},
  number  = {3},
  pages   = {357--393},
}

@Article{Itkin3D,
  Title                    = {LSV models with stochastic interest rates and correlated jumps},
  Author                   = {A. Itkin},
  Journal                  = {International Journal of Computer Mathematics},
  Year                     = {2017},
  volume                   = {94},
  number                   = {7},
  pages                    = {1291--1317},
}

@Article{OleskyEtAl2009,
  author  = {Olesky, Dale and Tsatsomeros, Michael and Van den Driessche, Pauline},
  title   = {Mv-matrices: a generalization of M-matrices based on eventually nonnegative matrices},
  journal = {The Electronic Journal of Linear Algebra},
  year    = {2009},
  volume  = {18},
  pages   = {339--351},
  doi     = {10.13001/1081-3810.1317}
}

@Article{NoutsosTsatsomeros2008,
  author  = {Noutsos, Dimitrios and Tsatsomeros, Michael J.},
  title   = {Reachability and holdability of nonnegative states},
  journal = {SIAM Journal on Matrix Analysis and Applications},
  year    = {2008},
  volume  = {30},
  number  = {2},
  pages   = {700--712},
  doi     = {10.1137/060678516}
}

@book{BermanPlemmons94,
  author    = {Berman, Abraham and Plemmons, Robert J.},
  title     = {Nonnegative Matrices in the Mathematical Sciences},
  publisher = {Society for Industrial and Applied Mathematics},
  address   = {Philadelphia, PA},
  series    = {Classics in Applied Mathematics},
  volume    = {9},
  year      = {1994},
  doi       = {10.1137/1.9781611971262},
  isbn      = {978-0-89871-321-3}
}

@book{HundsdorferVerwer2003,
  author    = {Hundsdorfer, Willem and Verwer, Jan G.},
  title     = {Numerical Solution of Time-Dependent Advection-Diffusion-Reaction Equations},
  series    = {Springer Series in Computational Mathematics},
  volume    = {33},
  publisher = {Springer},
  address   = {Berlin},
  year      = {2003}
}

@book{FCT2012,
  editor    = {Kuzmin, D. and L{\"o}hner, R. and Turek, S.},
  title     = {Flux-Corrected Transport: Principles, Algorithms, and Applications},
  edition   = {2nd},
  series    = {Scientific Computation},
  publisher = {Springer},
  address   = {Dordrecht},
  year      = {2012}
}

@Article{BolleyCrouzeix1978,
  author  = {Bolley, Catherine and Crouzeix, Michel},
  title   = {Conservation de la positivit{\'e} lors de la discr{\'e}tisation des probl{\`e}mes d'{\'e}volution paraboliques},
  journal = {RAIRO Analyse Num{\'e}rique},
  volume  = {12},
  number  = {3},
  year    = {1978},
  pages   = {237--245}
}

@Article{ScharfetterGummel1969,
  author  = {Scharfetter, D. L. and Gummel, H. K.},
  title   = {Large-signal analysis of a silicon {R}ead diode oscillator},
  journal = {IEEE Transactions on Electron Devices},
  volume  = {16},
  number  = {1},
  pages   = {64--77},
  year    = {1969}
}

@Article{LarsenEtAl1985,
  author  = {Larsen, E. W. and Levermore, C. D. and Pomraning, G. C. and Sanderson, J. G.},
  title   = {Discretization methods for one-dimensional {F}okker--{P}lanck operators},
  journal = {Journal of Computational Physics},
  volume  = {61},
  number  = {3},
  pages   = {359--390},
  year    = {1985}
}

@Article{BuetDellacherie2010,
  author  = {Buet, C. and Dellacherie, S.},
  title   = {On the {C}hang and {C}ooper scheme applied to a linear {F}okker--{P}lanck equation},
  journal = {Communications in Mathematical Sciences},
  volume  = {8},
  number  = {4},
  pages   = {1079--1090},
  year    = {2010}
}

@Article{MohammadiBorzi2015,
  author  = {Mohammadi, M. and Borz{\`i}, A.},
  title   = {Analysis of the {C}hang--{C}ooper discretization scheme for a class of {F}okker--{P}lanck equations},
  journal = {Journal of Numerical Mathematics},
  volume  = {23},
  number  = {3},
  pages   = {271--288},
  year    = {2015}
}

@Article{BessemoulinChatard2012,
  author  = {Bessemoulin-Chatard, M.},
  title   = {A finite volume scheme for convection--diffusion equations with nonlinear diffusion derived from the {S}charfetter--{G}ummel scheme},
  journal = {Numerische Mathematik},
  volume  = {121},
  number  = {4},
  pages   = {637--670},
  year    = {2012}
}

@Article{PareschiZanella2018,
  author  = {Pareschi, L. and Zanella, M.},
  title   = {Structure preserving schemes for nonlinear {F}okker--{P}lanck equations and applications},
  journal = {Journal of Scientific Computing},
  volume  = {74},
  number  = {3},
  pages   = {1575--1600},
  year    = {2018}
}

@Article{Sweby1984,
  author  = {Sweby, P. K.},
  title   = {High resolution schemes using flux limiters for hyperbolic conservation laws},
  journal = {SIAM Journal on Numerical Analysis},
  volume  = {21},
  number  = {5},
  pages   = {995--1011},
  year    = {1984}
}

@Article{KuzminTurek2002,
  author  = {Kuzmin, D. and Turek, S.},
  title   = {Flux correction tools for finite elements},
  journal = {Journal of Computational Physics},
  volume  = {175},
  number  = {2},
  pages   = {525--558},
  year    = {2002}
}

@Article{Kuzmin2009,
  author  = {Kuzmin, D.},
  title   = {Explicit and implicit {FEM-FCT} algorithms with flux linearization},
  journal = {Journal of Computational Physics},
  volume  = {228},
  number  = {7},
  pages   = {2517--2534},
  year    = {2009}
}

@Article{BarrenecheaJohnKnobloch2016,
  author  = {Barrenechea, G. R. and John, V. and Knobloch, P.},
  title   = {Analysis of algebraic flux correction schemes},
  journal = {SIAM Journal on Numerical Analysis},
  volume  = {54},
  number  = {4},
  pages   = {2427--2451},
  year    = {2016}
}

@Article{BarrenecheaJohnKnobloch2024,
  author  = {Barrenechea, G. R. and John, V. and Knobloch, P.},
  title   = {Finite element methods respecting the discrete maximum principle for convection--diffusion equations},
  journal = {SIAM Review},
  volume  = {66},
  number  = {1},
  pages   = {3--88},
  year    = {2024}
}

@Article{ZhangShu2010,
  author  = {Zhang, X. and Shu, C.-W.},
  title   = {On maximum-principle-satisfying high order schemes for scalar conservation laws},
  journal = {Journal of Computational Physics},
  volume  = {229},
  number  = {9},
  pages   = {3091--3120},
  year    = {2010}
}

@Article{GottliebShuTadmor2001,
  author  = {Gottlieb, S. and Shu, C.-W. and Tadmor, E.},
  title   = {Strong stability-preserving high-order time discretization methods},
  journal = {SIAM Review},
  volume  = {43},
  number  = {1},
  pages   = {89--112},
  year    = {2001}
}

@Article{KopeczMeister2018,
  author  = {Kopecz, S. and Meister, A.},
  title   = {On order conditions for modified {P}atankar--{R}unge--{K}utta schemes},
  journal = {Applied Numerical Mathematics},
  volume  = {123},
  pages   = {159--179},
  year    = {2018}
}

@Article{DanersGlueckKennedy2016,
  author  = {Daners, D. and Gl{\"u}ck, J. and Kennedy, J. B.},
  title   = {Eventually positive semigroups of linear operators},
  journal = {Journal of Mathematical Analysis and Applications},
  volume  = {433},
  number  = {2},
  pages   = {1561--1593},
  year    = {2016}
}

@Article{QiSun1993,
  author  = {Qi, L. and Sun, J.},
  title   = {A nonsmooth version of {N}ewton's method},
  journal = {Mathematical Programming},
  volume  = {58},
  pages   = {353--367},
  year    = {1993}
}

@Article{BurchardEtAl2003,
  author  = {Burchard, H. and Deleersnijder, E. and Meister, A.},
  title   = {A high-order conservative {P}atankar-type discretisation for stiff systems of production--destruction equations},
  journal = {Applied Numerical Mathematics},
  volume  = {47},
  number  = {1},
  pages   = {1--30},
  year    = {2003}
}

@Article{HintermuellerItoKunisch2002,
  author  = {Hinterm{\"u}ller, M. and Ito, K. and Kunisch, K.},
  title   = {The primal-dual active set strategy as a semismooth {N}ewton method},
  journal = {SIAM Journal on Optimization},
  volume  = {13},
  number  = {3},
  pages   = {865--888},
  year    = {2002}
}

@article{Zalesak1979,
  author  = {Zalesak, Steven T},
  title   = {Fully multidimensional flux-corrected transport algorithms for fluids},
  journal = {Journal of Computational Physics},
  year    = {1979},
  volume  = {31},
  pages   = {335--362},
  doi     = {10.1016/0021-9991(79)90051-2}
}

@unpublished{ItkinKazbek2026ADI,
  title={Diagonal Frog meets ADI: trading matrix exponentials for resolvents in the Fokker-Planck equation},
  author={Itkin, Andrey and Kazbek, Rakhymzhan},
  year={2026},
  note={in preparation}
}

@article{BorisBook1973,
  title={Flux-corrected transport. I. SHASTA, a fluid transport algorithm that works},
  author={Boris, Jay P. and Book, David L.},
  journal={Journal of Computational Physics},
  volume={11},
  number={1},
  pages={38--69},
  year={1973},
  publisher={Elsevier},
  doi={10.1016/0021-9991(73)90147-2}
}

@article{Godunov1959,
  title={A difference method for numerical calculation of discontinuous solutions of the equations of hydrodynamics},
  author={Godunov, S. K.},
  journal={Mat. Sb. (N.S.)},
  volume={47(89)},
  number={3},
  pages={271--306},
  year={1959}
}

@article{ItkinDF2026,
  title={Diagonal Frog: High-order positivity-preserving FD schemes for anisotropic Fokker-Planck equations},
  author={Itkin, Andrey},
  year={2026},
  eprint={2606.23980},
  archivePrefix={arXiv},
  primaryClass={math.NA}
}

@Book{ItkinLocalVol,
  author    = {A. Itkin},
  title     = {{Fitting Local Volatility: Analytic and Numerical Approaches in Black-Scholes and Local Variance Gamma Models}},
  number    = {11623},
  Year      = {2020},
  publisher = {World Scientific Publishing Co. Pte. Ltd.},
}

@article{Kuznetsov1976,
  author  = {Kuznetsov, N. N.},
  title   = {Accuracy of some approximate methods for computing the weak solutions of a first-order quasi-linear equation},
  journal = {USSR Computational Mathematics and Mathematical Physics},
  volume  = {16},
  number  = {6},
  pages   = {105--119},
  year    = {1976}
}

@article{Sabac1997,
  author  = {{\c{S}}abac, F.},
  title   = {The optimal convergence rate of monotone finite difference methods for hyperbolic conservation laws},
  journal = {SIAM Journal on Numerical Analysis},
  volume  = {34},
  number  = {6},
  pages   = {2306--2318},
  year    = {1997}
}

@article{ZhangShu2011,
  author  = {Zhang, X. and Shu, C.-W.},
  title   = {Maximum-principle-satisfying and positivity-preserving high-order schemes for conservation laws: survey and new developments},
  journal = {Proceedings of the Royal Society A: Mathematical, Physical and Engineering Sciences},
  volume  = {467},
  number  = {2134},
  pages   = {2752--2776},
  year    = {2011}
}

@article{Moser1964,
  author    = {J{\"u}rgen Moser},
  title     = {A Harnack inequality for parabolic differential equations},
  journal   = {Communications on Pure and Applied Mathematics},
  volume    = {17},
  number    = {1},
  pages     = {101--134},
  year      = {1964},
  publisher = {Wiley},
  doi       = {10.1002/cpa.3160170106}
}

\vspace{0.4in}
\appendixpage
\appendix
\numberwithin{equation}{section}
\setcounter{equation}{0}

\section{Proof of \cref{prop:schemeBp}} \label{appSchemeB}

Throughout, $\gamma>0$ is fixed, $M:=\mathcal I-\gamma A_1$, and $\bm b\ge0$. We write the limiter with the even budget split $\kappa_j\equiv2$, so that the caps assigned to interface $i+\tfrac12$ are
\begin{equation}\label{eq:caps}
c^{+}_{i+1/2} \;=\; \frac{h\,b_i}{2\gamma}, \qquad
c^{-}_{i+1/2} \;=\; \frac{h\,b_{i+1}}{2\gamma},
\end{equation}
the first bounding a positive flux, which withdraws from node $i$, and the second a negative flux, which withdraws from node $i+1$. The boundary interfaces carry no correction flux under the zero-flux closure.

\begin{lemma}[The core resolvent]\label{lem:core}
$M$ is a nonsingular M-matrix with unit column sums. Consequently $M^{-1}>0$ entrywise, $\bm1^\top M^{-1}=\bm1^\top$, and $\norm{M^{-1}}_1=1$ exactly.
\end{lemma}

\begin{proof}
$A_1$ is Metzler with $\bm1^\top A_1=0$, so $M$ has positive diagonal, nonpositive off-diagonal entries, and $\bm1^\top M=\bm1^\top$. Then $M^\top$ is a Z-matrix satisfying $M^\top\bm1=\bm1>0$, which makes $M^\top$, and hence $M$, a nonsingular M-matrix \cite{BermanPlemmons94}. Irreducibility of $A_1$ transfers to $M$, so $M^{-1}$ is strictly positive. Multiplying $\bm1^\top M=\bm1^\top$ by $M^{-1}$ on the right gives $\bm1^\top M^{-1}=\bm1^\top$, so every column of the nonnegative matrix $M^{-1}$ sums to one, and the induced $\ell_1$ norm, which is the maximal absolute column sum, equals one.
\end{proof}

\begin{lemma}[Clamp form of the limiter]\label{lem:clamp}
For any $\bm p\in\mathbb{R}^n$, the limited flux produced by \eqref{eq:zalesak} with the even split equals the clamp of the unlimited flux onto the fixed interval defined by the caps \eqref{eq:caps},
\begin{equation}\label{eq:clamp}
\theta_{i+1/2}\,d_{i+1/2}(\bm p) \;=\; \Lambda_{i+1/2}(\bm p)
\;:=\; \max\Bigl(-c^{-}_{i+1/2},\,\min\bigl(d_{i+1/2}(\bm p),\,c^{+}_{i+1/2}\bigr)\Bigr).
\end{equation}
The map $d\mapsto\Lambda$ is monotone and satisfies $|\Lambda(d)-\Lambda(d')|\le|d-d'|$ for any two flux values, since a clamp onto a fixed interval containing zero is nonexpansive.
\end{lemma}

\begin{proof}
For $d>0$ the rule \eqref{eq:zalesak} gives $\theta d=\min(d,c^{+})$, for $d<0$ it gives $\theta d=\max(d,-c^{-})$, and for $d=0$ both sides vanish. The interval $[-c^{-},c^{+}]$ contains zero because $\bm b\ge0$, and the clamp onto a fixed interval is nonexpansive and monotone by inspection.
\end{proof}

With \cref{lem:clamp} one sweep of the iteration \eqref{eq:picardlim}, with the limiter recomputed from the current iterate, is the map
\begin{equation}\label{eq:sweepmap}
\Phi(\bm p) \;=\; M^{-1}\bigl(\bm b+\gamma\,D_\Lambda(\bm p)\bigr),
\qquad
\bigl(D_\Lambda(\bm p)\bigr)_i \;=\; -\,\frac{\Lambda_{i+1/2}(\bm p)-\Lambda_{i-1/2}(\bm p)}{h},
\end{equation}
so that $\bm p^{[k+1]}=\Phi(\bm p^{[k]})$ and $\bm p^{[0]}=\bm b$.

\begin{proof}[Proof of \cref{prop:schemeBp}]
\emph{(i) Positivity.} We show that the right-hand side $\bm r(\bm p):=\bm b+\gamma D_\Lambda(\bm p)$ is nonnegative for \emph{every} $\bm p\in\mathbb{R}^n$, after which $\Phi(\bm p)=M^{-1}\bm r(\bm p)\ge0$ by \cref{lem:core}, and all iterates are nonnegative since $\bm p^{[0]}=\bm b\ge0$. Fix a node $i$. The interface $i+\tfrac12$ contributes $-\gamma\Lambda_{i+1/2}/h$ to $r_i$, which is negative only if $\Lambda_{i+1/2}>0$, in which case its magnitude is at most $\gamma c^{+}_{i+1/2}/h=b_i/2$ by \eqref{eq:clamp}. The interface $i-\tfrac12$ contributes $+\gamma\Lambda_{i-1/2}/h$, which is negative only if $\Lambda_{i-1/2}<0$, in which case its magnitude is at most $\gamma c^{-}_{i-1/2}/h=b_i/2$. All other contributions to $r_i$ are nonnegative deposits, so $r_i\ge b_i-\tfrac{b_i}{2}-\tfrac{b_i}{2}=0$. Note that no window in $\gamma$ was used, and that a node with $b_i=0$ has zero caps, so nothing is withdrawn from an empty node.

\emph{(ii) Conservation.} Each interface value $\Lambda_{i+1/2}$ enters $D_\Lambda$ exactly twice, in row $i$ with weight $-1/h$ and in row $i+1$ with weight $+1/h$, while the boundary interfaces carry no flux. Hence $\bm1^\top D_\Lambda(\bm p)=0$ identically, for every $\bm p$ and every value of the limiters, since scaling an interface term rescales both of its appearances together. Applying $\bm1^\top$ to $M\bm p^{[k+1]}=\bm r(\bm p^{[k]})$ and using $\bm1^\top M=\bm1^\top$ gives $\bm1^\top\bm p^{[k+1]}=\bm1^\top\bm b$, and the base case $\bm p^{[0]}=\bm b$ is immediate.

\emph{(iii) Contraction.} Let $\bm p,\bm p'\in\mathbb{R}^n$. By \cref{lem:core} and \eqref{eq:sweepmap},
\begin{align*}
\norm{\Phi(\bm p)-\Phi(\bm p')}_1 &\le \gamma\,\norm{D_\Lambda(\bm p)-D_\Lambda(\bm p')}_1 \;\le\; \frac{2\gamma}{h}\sum_{\text{interfaces}}\bigl|\Lambda_{i+1/2}(\bm p)-\Lambda_{i+1/2}(\bm p')\bigr| \\
&\le \frac{2\gamma}{h}\sum_{\text{interfaces}}\bigl|d_{i+1/2}(\bm p)-d_{i+1/2}(\bm p')\bigr|,
\end{align*}
the middle step because each interface difference appears in at most two rows, and the last step by the nonexpansiveness in \cref{lem:clamp}. The correction flux is a difference of two neighboring values with coefficients bounded by $\bar\mu/2$, so each nodal difference $|p_j-p'_j|$ enters the sum through at most two interfaces, each with weight at most $\bar\mu/2$, giving $\sum|d(\bm p)-d(\bm p')|\le\bar\mu\,\norm{\bm p-\bm p'}_1$. Altogether
\begin{equation}\label{eq:qbound}
\norm{\Phi(\bm p)-\Phi(\bm p')}_1 \;\le\; q\,\norm{\bm p-\bm p'}_1,
\qquad q\;=\;\frac{2\gamma\bar\mu}{h},
\end{equation}
so the constant in the proposition is $c=2$. For $q<1$ the map $\Phi$ is a contraction on $(\mathbb{R}^n,\norm{\cdot}_1)$ and possesses a unique fixed point $\bm p^{*}$ by the Banach theorem, the iterates converge to it geometrically, and $\bm p^{*}=\Phi(\bm p^{*})\ge0$ by part (i). The frozen-limiter variant is the linear map obtained by fixing $\theta$, and the same estimate applies since $\norm{C_\theta}_1\le\norm{C}_1\le2\bar\mu/h$ uniformly in $\theta\in[0,1]^{n-1}$. No parabolic term enters \eqref{eq:qbound} because $D$ appears only in $A_1$, whose resolvent was bounded by exactly one in \cref{lem:core}.

\emph{(iv) Conditional consistency.} We prove the two claims in turn. First, the exact statement. Suppose that at the fixed point the caps are slack at both interfaces adjacent to a node $i$,
\begin{equation}\label{eq:slack}
\gamma\,\bigl|d_{i\pm1/2}(\bm p^{*})\bigr| \;\le\; \frac{h\,b_{j}}{2},
\end{equation}
with $j$ the respective donor node. Then $\Lambda_{i\pm1/2}(\bm p^{*})=d_{i\pm1/2}(\bm p^{*})$ by \eqref{eq:clamp}, equivalently $\theta_{i\pm1/2}=1$, and row $i$ of the fixed-point equation $M\bm p^{*}=\bm b+\gamma D_\Lambda(\bm p^{*})$ coincides with row $i$ of the unlimited second-order system $(\mathcal I-\gamma A_2)\,\bm p=\bm b$. On any region where \eqref{eq:slack} holds at every interface, the converged iterate therefore satisfies the full second-order equations row by row. If \eqref{eq:slack} holds at every interface of the grid, then $\bm p^{*}$ is exactly the unlimited solution, since the latter is then a fixed point of $\Phi$ and the fixed point is unique.

Second, the sufficient condition. Suppose that on a set of interfaces the fixed point satisfies the log-Lipschitz resolution condition of \cite{ItkinDF2026} with constant $\Lambda_p$, so that neighboring values are comparable, $\max(p^{*}_i,p^{*}_{i-1})\le e^{\Lambda_p h}\min(p^{*}_i,p^{*}_{i-1})$, and let $\mu$ be Lipschitz with constant $L_\mu$. Assuming $\mu_i > 0$ without loss of generality (the $\mu_i < 0$ case being symmetric on $\{i+1, i+2\}$), writing $m := \min(p^{}i, p^{*}{i-1})$, we have
\begin{equation*}
\bigl|d_{i+1/2}(\bm p^{*})\bigr|
\;=\;\tfrac12\bigl|(\mu p^{*})_i-(\mu p^{*})_{i-1}\bigr|
\;\le\;\tfrac12\Bigl[\bar\mu\,\bigl(e^{\Lambda_p h}-1\bigr)+L_\mu h\,e^{\Lambda_p h}\Bigr]\,m
\;\le\;\tfrac{h}{2}\,\bigl(\bar\mu\Lambda_p+L_\mu\bigr)\,e^{\Lambda_p h}\,m .
\end{equation*}
Since $m\le p^{*}_{j}$ for the donor node $j$, the slack condition \eqref{eq:slack} holds whenever
\begin{equation}\label{eq:suffslack}
\gamma\,\bigl(\bar\mu\Lambda_p+L_\mu\bigr)\,e^{\Lambda_p h}\; p^{*}_{j} \;\le\; b_{j},
\end{equation}
which is satisfied for all sufficiently small $\gamma$ at every node where $b_j$ and $p^{*}_j$ are comparable, and in particular fails only where the density is at the level of the local truncation error, since there the correction flux is dominated by the discretization error itself rather than by the resolved profile. This proves the localization claim.

We take \eqref{eq:suffslack} as the a posteriori form of the hypothesis. It is computable from the converged iterate, so the localization claim is checked directly on each run rather than assumed, and \cref{ssec:numPeclet} reports that check. Turning it into an a priori statement would require a Harnack-type comparison of $b_j$ with $p^{*}_j$ through the resolvent of $A_1$, which we leave open.

\end{proof}

\section{Proof of \cref{prop:peclet}} \label{appPeclet}

Throughout, $M=\mathcal I-\gamma A_1$, $q=2\gamma\bar\mu/h\le\kappa<1$, and $\norm{\cdot}_1$ denotes the unweighted vector norm, so that $\norm{\bm v}_{1,h}=h\norm{\bm v}_1$. We use three facts already established: $\norm{M^{-1}}_1=1$ (\cref{lem:core}), the sweep maps $\Phi(\bm p)=M^{-1}(\bm b+\gamma D_\Lambda(\bm p))$ and $\Phi_\infty(\bm p)=M^{-1}(\bm b+\gamma C\bm p)$ are Lipschitz in $\ell_1$ with constant $q$ (\cref{prop:schemeBp}(iii), the unlimited map being the linear case), and $\bm p^{*}$, $\hat{\bm p}$ are their respective fixed points. The error splits as
\begin{equation}\label{eq:errsplit}
\bm p^{*}-\bm u \;=\; \underbrace{(\hat{\bm p}-\bm u)}_{\text{consistency}} \;+\; \underbrace{(\bm p^{*}-\hat{\bm p})}_{\text{limiter defect}} .
\end{equation}

\emph{Step 1: stability.} If $\bm e$ satisfies $(\mathcal I-\gamma A_2)\,\bm e=\gamma\bm\tau$, then $\bm e=M^{-1}(\gamma C\bm e+\gamma\bm\tau)$, hence $\norm{\bm e}_1\le q\norm{\bm e}_1+\gamma\norm{\bm\tau}_1$ and
\begin{equation}\label{eq:stab}
\norm{\bm e}_1 \;\le\; \frac{\gamma}{1-\kappa}\,\norm{\bm\tau}_1 .
\end{equation}
No stability assumption on $A_2$ beyond the split is used: the bound is inherited from the M-matrix core through the contraction.

\emph{Step 2: consistency of the unlimited solution.} Evaluating $(1-\gamma\mathcal{L})u=b$ at the nodes and subtracting from $(\mathcal I-\gamma A_2)\hat{\bm p}=\bm b$ gives $(\mathcal I-\gamma A_2)(\hat{\bm p}-\bm u)=\gamma\bm\tau$ with the truncation error $\bm\tau=A_2\bm u-(\mathcal{L}u)|_{\mathrm{grid}}$. Both terms of $\bm\tau$ are flux differences. The discrete operator telescopes with the numerical fluxes of \cref{sec:setting}, and the differential operator satisfies the midpoint identity $(\mathcal{L}u)(x_i)=-[J(x_{i+1/2})-J(x_{i-1/2})]/h+r_i$ with $|r_i|\le\tfrac{h^2}{24}\max|J'''|$ on smooth cells. Hence
\begin{equation}\label{eq:tauflux}
\tau_i \;=\; -\,\frac{\varepsilon_{i+1/2}-\varepsilon_{i-1/2}}{h} \;-\; r_i ,
\qquad
\varepsilon_{i+1/2} \;:=\; \hat J_{i+1/2}(\bm u)-J(x_{i+1/2}),
\end{equation}
where $\hat J$ is the total numerical flux, advective minus diffusive. On a smooth interface, Taylor expansion of the second-order upwind advective flux gives $|\hat g^{(2)}_{i+1/2}-(\mu u)(x_{i+1/2})|\le\tfrac{3}{8}h^2\max|(\mu u)''|$, and of the centered diffusive flux $\bigl|[(Du)_{i+1}-(Du)_i]/h-(Du)'(x_{i+1/2})\bigr|\le\tfrac{1}{24}h^2\max|(Du)'''|$, so $|\varepsilon_{i+1/2}|\le h^2 M_S$ there. On a layer interface we bound crudely by the fluxes themselves, $|\varepsilon_{i+1/2}|\le|\hat J_{i+1/2}(\bm u)|+|J(x_{i+1/2})|\le C_1\bar\mu S_L$, using the third assumption, and the near-boundary first-order rows contribute $O(h)$ flux errors at $O(1)$ interfaces, absorbed into the same term. Since each interface error enters \eqref{eq:tauflux} in at most two rows,
\begin{equation}\label{eq:taubound}
\norm{\bm\tau}_1 \;\le\; \frac{2}{h}\sum_{\text{interfaces}}|\varepsilon| \;+\; \sum_i|r_i|
\;\le\; \frac{2}{h}\Bigl[\frac{|\Omega|}{h}\,h^2M_S + C_1N_L\bar\mu S_L\Bigr] + \frac{|\Omega|}{h}\,\frac{h^2}{24}M_S .
\end{equation}
Combining \eqref{eq:stab}, \eqref{eq:taubound}, $\gamma\le h/(2\bar\mu)$, and $\norm{\cdot}_{1,h}=h\norm{\cdot}_1$,
\begin{equation}\label{eq:consbound}
\norm{\hat{\bm p}-\bm u}_{1,h}
\;\le\; \frac{C_2}{1-\kappa}\Bigl[\,\frac{|\Omega|}{\bar\mu}\,M_S\,h^2 + N_LS_L\,h\,\Bigr].
\end{equation}

\emph{Step 3: the limiter defect.} Since $\Phi$ is a $q$-contraction with fixed point $\bm p^{*}$, the standard perturbation bound gives $\norm{\hat{\bm p}-\bm p^{*}}_1\le(1-q)^{-1}\norm{\Phi(\hat{\bm p})-\hat{\bm p}}_1$. Using $\hat{\bm p}=\Phi_\infty(\hat{\bm p})$,
\begin{equation*}
\Phi(\hat{\bm p})-\hat{\bm p} \;=\; \gamma\,M^{-1}\bigl(D_\Lambda(\hat{\bm p})-C\hat{\bm p}\bigr),
\end{equation*}
and the argument of $M^{-1}$ is a flux difference of the clamp residuals $\Lambda_{i+1/2}(\hat{\bm p})-d_{i+1/2}(\hat{\bm p})$, which vanish off the active set of $\hat{\bm p}$ and are bounded by $|d_{i+1/2}(\hat{\bm p})|$ on it. By the second and third assumptions the active set lies in $\mathcal A$ and the fluxes there are bounded by $\bar\mu S_L$, so
\begin{equation}\label{eq:defectbound}
\norm{\bm p^{*}-\hat{\bm p}}_{1,h}
\;\le\; \frac{h}{1-\kappa}\cdot\gamma\cdot\frac{2}{h}\cdot N_L\,\bar\mu S_L\cdot h \cdot \frac1h \;\le\; \frac{1}{1-\kappa}\,N_LS_L\,h,
\end{equation}
again by $2\gamma\bar\mu\le h$. Adding \eqref{eq:consbound} and \eqref{eq:defectbound} through \eqref{eq:errsplit} yields \eqref{eq:pecletbound} with $C(\kappa)=C_3/(1-\kappa)$. Every constant is independent of $D$ and of $\mathrm{Pe}_h$: the diffusion coefficient enters the argument only through $M^{-1}$, whose norm is exactly one, and through $M_S$, which the second assumption bounds uniformly. Part (a) follows by inserting $S_L\le C_0h$, and part (b) by inspection of the two terms.

\section{Proof of \cref{prop:dc}} \label{appDC}

We first extend \cref{lem:clamp} to a clamp with a fixed offset.

\begin{lemma}[Clamp with fixed additional fluxes]\label{lem:clampoffset}
Fix an interface and abbreviate $c^{+}=h\,p^{\,n}_i/2$, $c^{-}=h\,p^{\,n}_{i+1}/2$ and $G=G_{i+1/2}$. The map $\bm p\mapsto\widetilde\Lambda_{i+1/2}(\bm p)$ of \eqref{eq:combflux} is the composition of the translation $d\mapsto \Delta t\,d+G$ with the clamp onto $[-c^{-},c^{+}]$. Since the translation is an isometry and the clamp onto a fixed interval is nonexpansive and monotone, for any $\bm p,\bm p'$,
\begin{equation*}
\bigl|\widetilde\Lambda_{i+1/2}(\bm p)-\widetilde\Lambda_{i+1/2}(\bm p')\bigr|
\;\le\;\Delta t\,\bigl|d_{i+1/2}(\bm p)-d_{i+1/2}(\bm p')\bigr| ,
\end{equation*}
and the value $\widetilde\Lambda_{i+1/2}(\bm p)$ always lies in $[-c^{-},c^{+}]$, an interval containing zero because $\bm p^{\,n}\ge0$.
\end{lemma}

\begin{proof}
Immediate from the two displayed properties of the clamp, exactly as in \cref{lem:clamp}. The offset $G$ shifts the argument of the clamp but not the interval, so it cancels in the difference and does not enlarge the range.
\end{proof}

\begin{proof}[Proof of \cref{prop:dc}]
Write $M=\mathcal I-\Delta t\,A_1$ and let $\Psi$ denote the sweep map of a stage,
\begin{equation*}
\Psi(\bm p)\;=\;M^{-1}\Bigl(\bm p^{\,n}-\tfrac1h\bigl[\widetilde\Lambda_{i+1/2}(\bm p)-\widetilde\Lambda_{i-1/2}(\bm p)\bigr]_i\Bigr),
\end{equation*}
with $G\equiv0$ for the predictor and $G$ as in \cref{ssec:dclimited} for the corrector. \Cref{lem:core} applies verbatim to $M$ with $\gamma=\Delta t$.

\emph{(i) Positivity.} The right-hand side of $\Psi$ is nonnegative for every input $\bm p$: at node $i$, the interface $i+\tfrac12$ withdraws at most $\widetilde\Lambda_{i+1/2}/h\le c^{+}/h=p^{\,n}_i/2$ when its value is positive, the interface $i-\tfrac12$ withdraws at most $c^{-}_{i-1/2}/h=p^{\,n}_i/2$ (by \eqref{eq:caps} applied to $i-\tfrac12$ with $\bm b = \bm p^{\,n}$) when its value is negative, all other contributions are deposits, and the range statement of \cref{lem:clampoffset} guarantees these bounds regardless of $G$. Hence $r_i\ge p^{\,n}_i-\tfrac{p^{\,n}_i}{2}-\tfrac{p^{\,n}_i}{2}=0$, and $\Psi(\bm p)=M^{-1}\bm r\ge0$. Both stages start from $\bm p^{[0]}=\bm p^{\,n}\ge0$, so every iterate and both fixed points are nonnegative.

\emph{(ii) Conservation and stability.} Each clamped value enters the divergence in exactly two rows with opposite signs and the boundary interfaces carry no flux, so $\bm1^\top$ of the flux divergence vanishes identically, for every limiter value and every $G$. Applying $\bm1^\top$ to $M\bm p^{[k+1]}=\bm r(\bm p^{[k]})$ and using $\bm1^\top M=\bm1^\top$ gives $\bm1^\top\bm p^{[k+1]}=\bm1^\top\bm p^{\,n}$ at every sweep of both stages, hence $\bm1^\top\bm p^{\,n+1}=\bm1^\top\bm p^{\,n}$. On nonnegative vectors the $\ell_1$ norm equals the mass, so $\norm{\bm p^{\,n+1}}_1=\norm{\bm p^{\,n}}_1$ and the map from $\bm p^{\,n}$ to $\bm p^{\,n+1}$ is $\ell_1$-isometric on the cone, which is unconditional stability.

\emph{(iii) Contraction.} By \cref{lem:core}, \cref{lem:clampoffset}, and the same two-row counting as in the proof of \cref{prop:schemeBp}(iii),
\begin{equation*}
\norm{\Psi(\bm p)-\Psi(\bm p')}_1
\;\le\;\frac{2}{h}\sum_{\text{interfaces}}\bigl|\widetilde\Lambda(\bm p)-\widetilde\Lambda(\bm p')\bigr|
\;\le\;\frac{2\Delta t}{h}\sum_{\text{interfaces}}\bigl|d(\bm p)-d(\bm p')\bigr|
\;\le\;\frac{2\Delta t\,\bar\mu}{h}\,\norm{\bm p-\bm p'}_1 ,
\end{equation*}
the offsets $G$ having cancelled in the first inequality. For $q=2\Delta t\bar\mu/h<1$ the Banach theorem gives the unique fixed point of each stage and geometric convergence.

\emph{(iv) Temporal order.} Suppose the caps are slack at the fixed points of both stages on a region $\mathcal R$ of nodes, meaning every combined flux at the interfaces of $\mathcal R$ satisfies $|F|\le\min(c^{+},c^{-})$ there. Then $\widetilde\Lambda=F$ on those interfaces, and the rows of the two fixed-point systems on $\mathcal R$ coincide with the rows of \eqref{eq:predictor} and \eqref{eq:corrector}, since $\Delta t\,d+G$ unclamped reassembles $\Delta t\,C\bm p+\bm s$ and $M\bm p-\Delta t\,C\bm p=(\mathcal I-\Delta t A_2)\bm p$ by the split \eqref{eq:coresplit}. On $\mathcal R$ the computed step is therefore the linear two-stage scheme, whose stability function \eqref{eq:Rfunc} matches $e^z$ through $z^2$, giving a local error $O(\Delta t^3)$ and, combined with the $O(h^2)$ spatial consistency of $A_2$, a second-order step in time and space. If a cap binds at an interface, the clamp replaces the combined flux there by a budget value, which is the local fallback toward the limited backward Euler right-hand side, first-order in time. The binding condition $|F|>h\,p^{\,n}_j/2$ requires the combined flux to exceed half the local density scale, which by the estimates in the proof of \cref{prop:schemeBp}(iv) can occur only where the density is at the level of the truncation error or inside an unresolved layer, the same set on which the spatial limiter is active.
\end{proof}

\section{Proof of \cref{lem:separated}} \label{app:separated}

\begin{proof}
Let $\mathcal I_S$ be the set of free interfaces of $S$. The correction operator $C_S$ acts nontrivially only on $\mathcal I_S$; clamped interfaces contribute zero rows and columns to $C_S$. After a symmetric permutation that orders the unknowns block-wise as
\begin{equation}\label{eq:blockorder}
\text{(free block 1, clamped block 1, free block 2, clamped block 2, \dots, free block $k$, clamped block $k$)},
\end{equation}
the matrix $V_S$ takes the block-diagonal form
\begin{equation}\label{eq:blockdiag}
V_S = \begin{pmatrix}
V_{F_1} & 0      & \cdots & 0 \\
0      & M_{C_1} & \cdots & 0 \\
\vdots & \vdots  & \ddots & \vdots \\
0      & 0       & \cdots & V_{F_k} \\
0      & 0       & \cdots & M_{C_k}
\end{pmatrix},
\end{equation}
where each $V_{F_i}$ is the restriction of $\mathcal I-\gamma(A_1+C_S)$ to the $i$-th free block and its immediate neighbors (which include the bounding clamped interfaces in the stencil), and each $M_{C_i}$ is the restriction of $M$ to the $i$-th clamped block. The separation condition ensures that the correction stencils of distinct free blocks do not overlap: the support of $C_S$ on free block $i$ extends at most to the adjacent clamped interfaces, which belong to the boundary of that block and do not couple to the interior of the next free block. Consequently, the off-diagonal coupling between distinct free blocks vanishes, and the permutation yields the block-diagonal structure \eqref{eq:blockdiag}.

Each diagonal block $V_{F_i}$ is the restriction of $\mathcal I-\gamma(A_1+C_S)$ to a subdomain consisting of a contiguous free segment bounded by clamped interfaces. On this subdomain the correction $C_S$ coincides with the all-free correction restricted to that segment, and the norm estimate of part (a) applies locally because the segment length is finite and the boundary conditions at the clamped interfaces are Dirichlet-type (fixed flux values). By the same Neumann series argument, each $V_{F_i}$ is nonsingular. Each clamped block $M_{C_i}$ is a principal submatrix of the monotone core matrix $M = \mathcal I - \gamma A_1$. Because $M$ is a nonsingular M-matrix (as established in \cref{lem:core}), its principal submatrix $M_{C_i}$ is also a nonsingular M-matrix. Nonsingularity of $V_S$ then follows from the nonsingularity of all its diagonal blocks.

\end{proof}

\end{document}